\documentclass[11pt]{article}

\usepackage{amssymb}
\input{WSamnm_mainmac2}
\input{WSamnm_forme}
\usepackage{hyperref}

\usepackage[usenames]{color}

\title{\vskip-1.0em\sc Approximately multiplicative maps from weighted semilattice algebras}
\author{\sc Yemon Choi}
\date{6th May 2013}

\begin{document}

\maketitle

\begin{abstract}
We investigate which weighted convolution algebras $\lom(S)$, where $S$ is a semilattice, are AMNM in the sense of Johnson (JLMS, 1986). We give an explicit example where this is not the case. We show that the unweighted examples are all AMNM, as are all $\lom(S)$ where $S$ has either finite width or finite height. Some of these finite-width examples are isomorphic to function algebras studied by Feinstein (IJMMS, 1999).

We also investigate when $(\lom(S),\Mat{2})$ is an AMNM pair in the sense of Johnson (JLMS, 1988), where $\Mat{2}$ denotes the algebra of $2$-by-$2$ complex matrices. In particular, we obtain the following two contrasting results: (i)~for many non-trivial weights on the totally ordered semilattice $\Nmin$, the pair $(\lom(\Nmin),\Mat{2})$ is not AMNM; (ii)~for any semilattice $S$, the pair $(\ell^1(S),\Mat{2})$ is AMNM.
The latter result requires a detailed analysis of approximately commuting, approximately idempotent $2\times 2$ matrices.

\YCrem{``Systems of functional equations and inequalities''; used by e.g. Hyers and Rassias}
46J10 (secondary)
\YCrem{``Banach algebras of continuous functions, function algebras''; used in Feinstein's article, and Sidney's}

\end{abstract}


\begin{section}{Introduction}

\subsection{Setting the scene}
Given a constant $\delta>0$, we say that a functional $\psi$ on a Banach algebra $A$ is $\delta$-multiplicative if the bilinear map $(a,b) \mapsto \psi(a)\psi(b)-\psi(ab)$ has norm at most $\delta$.
It is convenient, thinking of $\delta$ as small, to call such functionals \dt{approximately multiplicative} or \dt{almost multiplicative} (we shall use the former phrase).
Approximately multiplicative functionals have been studied by several authors: an obvious way to obtain examples is to take a multiplicative functional and add a functional of small norm, thought of as a perturbation.
The question naturally arises as to whether \emph{all} approximately multiplicative functionals occur in this way.

In \cite{BEJ_AMNM1}, B. E. Johnson undertook a  systematic study of this phenomenon, and coined the acronym AMNM (for Approximately Multiplicative implies Near Multiplicative). The precise definition will be deferred to a later section. Many examples of commutative Banach algebras with the AMNM property are given in \cite{BEJ_AMNM1}, as are some basic hereditary properties.
See also~\cite{Jarosz_AMNM,Sidney_AMNM} for results on uniform algebras, and \cite{Howey_JLMS} for results on certain non-uniform function algebras, including~$C^k[0,1]^m$.
The paper \cite{BEJ_AMNM2} widens the scope of the problem, by considering not just functionals, but approximately multiplicative linear maps between given Banach algebras. This leads to the notion of an \dt{AMNM pair}; again, the precise definition will be given below. We note that, as a special case of \cite[Theorem 3.1]{BEJ_AMNM2}, every amenable Banach algebra has the AMNM property; however, several of the examples in \cite{Howey_JLMS,Jarosz_AMNM,BEJ_AMNM1} possess non-zero point derivations, so that amenability is far from necessary for AMNM.

In this paper, we investigate these AMNM problems for the weighted $\ell^1$-convolution algebras of semilattices. Such algebras have provided provide useful test cases for various conjectures and techniques concerning commutative Banach algebras.
Moreover, any weighted semilattice algebra contains a dense subalgebra spanned by commuting idempotents. Thus, some of our work can be viewed as continuing an old strand of Banach algebra theory, which considers lifting and perturbation questions for families of idempotents.
The main difference here is that we are not limiting ourselves to families of pairwise orthogonal idempotents, but allowing more complicated order structure.

The original motivation for the present work arises from studying the cases where the underlying semilattice is $\Nmin$, the set of natural numbers equipped with pairwise minimum as a semigroup product. The weighted $\ell^1$-algebras of $\Nmin$ turn out to be isomorphic to function algebras that were studied by J.~F.~Feinstein in \cite{Fein_alg}; they have also been studied in the context of certain generalized notions of amenability, see for instance \cite[\S3.10]{DalesLoy_diss}. Moreover, some of these algebras satisfy such versions of amenability while having non-trivial 2nd-degree simplicial cohomology (the present author, unpublished calculations).

\subsection{Overview of the paper}
We have tried to make this paper self-contained, save for some basic knowledge of Banach algebras. Thus, in Section~\ref{s:preliminaries} we give the relevant definitions of the AMNM property for algebras and for pairs of algebras, as promised earlier; and we record some basic observations on convolution algebras and their characters.
We then observe that $\ell^1(S)$ is AMNM for any semilattice $S$ (Theorem~\ref{t:AMNM-no-weight}).
On the other hand, we give an explicit example of a semilattice $T$ and a weight on $T$ such that the weighted convolution algebra $\lom(T)$ is not AMNM (Theorem~\ref{t:not-AMNM}).
In Section~\ref{s:AMNM-examples}, as a special case of a general technical result, we prove that if $S$ has either finite width or finite height, then $\lom(S)$ is AMNM for every weight~$\om$. This applies in particular when $S=\Nmin$, the original case of interest.

The picture is far less complete if we consider approximately multiplicative maps into algebras other than~$\Cplx$. Let $\sT_2$ be the (commutative, non-semisimple) algebra of \dt{dual numbers} over~$\Cplx$, and let  $\Mat{2}$ be the (non-commutative, semisimple) algebra of $2\times 2$ matrices with entries in~$\Cplx$.
In Theorem~\ref{t:wt-Nmin-T2}, we show that whenever $\om$ is a non-trivial weight on $\Nmin$, then $(\lom(\Nmin),\sT_2)$ is not an AMNM pair.
In Theorem~\ref{t:wt-Nmin-M2}, we show that for many non-trivial weights on $\Nmin$, the pair $(\lom(\Nmin),\Mat{2})$ fails to be AMNM.

These examples suggest that, if we want positive AMNM results for range algebras other than $\Cplx$, we should focus attention on the unweighted case. Indeed, we prove that for an arbitrary semilattice $S$, the pairs $(\ell^1(S), \sT_2)$ and $(\ell^1(S),\Mat{2})$ are ``uniformly AMNM'' (the terminology is explained below, in Definition~\ref{d:unif-AMNM}). The proof of this for $\Mat{2}$ 
takes up all of Section~\ref{s:AMNM-M2}: although the techniques used are elementary, a complete proof seems to require substantially more work than is needed for $\sT_2$.
Finally, we close the paper by briefly discussing some possible avenues for future work.

\begin{rem}
Some of our calculations would work for certain weighted algebras on abelian Clifford semigroups, specifically those where the weight is trivial on each group component.
We have decided to only consider the semilattice case for now:
an adequate treatment of weighted abelian Clifford semigroups would require a more detailed look at AMNM problems for Beurling algebras than the present paper can accommodate.
 The same remarks apply for inverse semigroups: we did not see how to go beyond  superficial generalizations of the results here.
\end{rem}

\begin{rem}\label{r:belated}
After completing the work presented in this article, we learned of the paper~\cite{Law_amnm-to-Mn}, which considers a related but different notion of AMNM (briefly, the order of quantifiers is different). Although the stability result that is proved in [Theorems 2 and 5, ibid.] is different from ours, it may be possible to use those arguments to streamline the approach taken in Section~\ref{s:AMNM-M2}, and perhaps even to extend the results of Section~\ref{s:AMNM-M2} from $\Mat{2}$ to~$\Mat{n}$.
\end{rem}
\end{section}

\begin{section}{Preliminaries}\label{s:preliminaries}
We assume familiarity with the basics of Banach spaces and Banach algebras. Throughout $\ptp$ denotes the projective tensor product of Banach spaces. If $A$ is a Banach algebra then $\pi_A:A\ptp A\to A$ denotes the unique bounded bilinear map satisfying $\pi_A(a_1\tp a_2) = a_1a_2$.

\subsection{Defining AMNM}
Our notation is different from that of Johnson's articles~\cite{BEJ_AMNM1,BEJ_AMNM2}, and so we repeat some of the basic definitions for sake of clarity.

\begin{dfn}[Multiplicative defect]
Let $A$ and $B$ be Banach algebras, and let $T:A\to B$ be a bounded linear map. We define the \dt{multiplicative defect of $T$} to be
\begin{equation}\label{eq:diff}
\begin{aligned}
\diff(T)
 & = \norm{T\circ\pi_A - \pi_B \circ (T \tp T) : A\ptp A \to B} \\
 & = \sup \{ \norm{T(xy)-T(x)T(y)} \st x, y\in A, \norm{x}\leq 1 , \norm{y}\leq 1\} .
\end{aligned}
\end{equation}
\end{dfn}

$T$ is said to be \dt{$\delta$-multiplicative}, in the sense of \cite{BEJ_AMNM2}, when $\diff(T)\leq \delta$. Note that Johnson uses the notation $T^\vee$ instead of $\diff(T)$. 
Of course, a $0$-multiplicative map is just one that is multiplicative in the usual sense, that is, a continuous linear algebra \hm\ (which need not be unital, even if $A$ and $B$ are unital algebras). We denote by $\Mult(A,B)$ the set of multiplicative, bounded linear maps $A\to B$; a non-zero element of $\Mult(A,\Cplx)$ is called a \dt{character} of~$A$.

Given a subset $K$ in a metric space $(X,d)$ and $y\in X$,
 define $\dist_X(y,K)$ to be $\inf_{x\in X} d(x,y)$.

\begin{dfn}[AMNM algebras, \cite{BEJ_AMNM1}]\label{d:AMNM-alg}
A Banach algebra $A$ is said to be \dt{AMNM}, or have the \dt{AMNM property}, or have \dt{stable characters}, if for each $\veps>0$ there exists $\delta>0$ such that
\[ \psi\in A^*, \;\diff(\psi) \leq \delta \implies \dist_{A^*}(\psi, \Mult(A,\Cplx))\leq \veps. \]
\end{dfn}

While Definition \ref{d:AMNM-alg} does not require $A$ to be commutative, it seems most natural to study the AMNM property for commutative Banach algebras, since those are the ones for whom characters are most informative (via Gelfand theory). For non-commutative algebras~$A$, the following definition seems more natural.

\begin{dfn}[AMNM pairs of Banach algebras, \cite{BEJ_AMNM2}]
\label{d:AMNM-pairs}
Let $A$ and $B$ be Banach algebras. We say that the pair $(A,B)$ is an \dt{AMNM pair} if, for every $K>0$ and $\veps>0$, there exists $\delta>0$ such that
\[ T\in\Lin(A,B), \; \norm{T}\leq K,\; \diff(T) \leq \delta \implies \dist_{\Lin(A,B)}(T,\Mult(A,B)) \leq \veps. \]
\end{dfn}

The presence of the {\it a~priori}\/ upper bound $K$ may seem curious at first sight. One reason for imposing such a bound is that we can find $T$ such that $\dist_{\Lin(A,B)}(T,\Mult(A,B))$ is arbitrarily small while $\diff(T)\geq 1$. See \cite[p.~295]{BEJ_AMNM2} for an example with $A=\Cplx$ and $B$ the algebra $\Mat{2}$ of $2\times 2$ matrices.

Nevertheless, in the present paper, we shall obtain examples of pairs of algebras which are not just AMNM, but which satisfy a stronger property, defined as follows.

\begin{dfn}\label{d:unif-AMNM}
Let $A$ and $B$ be Banach algebras. We say that $(A,B)$ is a \dt{uniformly AMNM pair} if, for every $\veps>0$, there exists $\delta>0$ such that
\[ T\in\Lin(A,B), \; \diff(T) \leq \delta \implies \dist_{\Lin(A,B)}(T,\Mult(A,B)) \leq \veps. \]
\end{dfn}

With this terminology, a Banach algebra $A$ is AMNM if and only if the pair $(A,\Cplx)$ is uniformly AMNM.  In some cases, AMNM pairs are automatically uniformly AMNM: for instance, this is the case when $B=C(X)$, since a $\delta$-multiplicative linear map $A\to C(X)$ has norm at most $1+\delta$ \cite[Proposition 5.5]{Jarosz_LNM}. 
Note that Johnson~\cite[Example 1.5]{BEJ_AMNM2} has given an example of a commutative, semisimple Banach algebra $B$ for which the pair $(\Cplx,B)$ is AMNM but not uniformly AMNM.

\subsection{Some notation and terminology}
A \dt{weight} on a set $S$ is a function $\om: S \to (0,\infty)$. Given such a weight $\om$, we write $\lom(S)$ for the corresponding weighted $\ell^1$-space.
Throughout this paper, whenever we refer to a weight on a semigroup, we always mean a \emph{submultiplicative} weight. By a \dt{weighted semigroup}, we mean a pair $(S,\om)$ where $S$ is a semigroup and $\om$ is a weight on~$S$. A routine calculation shows that if $(S,\om)$ is a weighted semigroup, then $\lom(S)$, equipped with the natural convolution product, is a Banach algebra, called the \dt{weighted semigroup algebra} or \dt{convolution algebra} of $(S,\om)$.

A bounded function $f$ from a set $S$ to a Banach space $B$ has a unique continuous extension to a bounded linear map $\ell^1(S)\to B$, whose norm is precisely $\sup_{t\in S} \norm{f(t)}$. Moreover, if $S$ is a semigroup and $B$ is a Banach algebra, this extension will be multiplicative if and only if the original function $f$ is multiplicative.
Thus, the characters of $\ell^1(S)$ are in natural bijection with the non-zero semigroup homomorphisms $S\to (\Cplx,\times)$, sometimes called the \dt{semi-characters of~$S$}.

The corresponding version for weighted semigroup algebras is equally straightforward: bounded linear maps $\lom(S)\to B$ correspond to functions $f:S\to B$ such that
\begin{equation}\label{eq:above}
\sup_{s\in S} \om(s)^{-1}\norm{f(s)} <\infty,
\end{equation}
and so forth. Such a function will be called an \dt{$\om$-bounded map} from $S$ to $B$. In the case where $B=\Cplx$ and $f:S\to \Cplx$, the quantity in \eqref{eq:above} will be denoted by $\norm{f}_{\infty,\om^{-1}}$.

Since we can naturally and isometrically identify $\lom(S)\ptp\lom(S)$ with $\ell^1_{\om\times\om}(S\times S)$, it is easy to check that for a given function $f:S\to B$, the multiplicative defect of the corresponding linear map $\lom(S)\to B$ is
\begin{equation}\label{eq:defect-formula}
\diff_\om(f) \defeq \sup_{x,y\in S} \frac{\norm{f(x)f(y)-f(xy)}}{\om(x)\om(y)}. \end{equation}
In this way, all AMNM questions where the domain algebra $A$ is a weighted semigroup algebra can be rephrased in terms of $\om$-bounded maps from the given semigroup and $(\om\times\om)$-bounded maps from its Cartesian square. This is sometimes convenient if we need to define a map via case-by-case checking.

\subsection{Weighted semilattice algebras, and their characters}
A \dt{semilattice} is a commutative semigroup in which each element is idempotent.
Even quite simple semilattices can give rise to interesting Banach algebras, once we allow for weights. The following examples provided the initial motivation for this article, and will be revisited in Theorems \ref{t:wt-Nmin-T2} and \ref{t:wt-Nmin-M2}. 

\begin{eg}[``Feinstein algebras'']\label{eg:fein-alg}
Let $\Nmin$ denote the semilattice obtained by equipping the set of natural numbers with the binary operation $(m,n)\mapsto \min(m,n)$. The weights on this semilattice are precisely the functions $\om:\Nat\to [1,\infty)$. The convolution algebra $\lom(\Nmin)$ is semisimple and its character space can be naturally identified with~$\Nat$.
 (This is alluded to, without details, in \cite[11.1.5]{HewZuck_TAMS56}; see also Lemma~\ref{l:which-sets} below.)
In fact, the Gelfand transform maps $\lom(\Nmin)$ onto a dense subalgebra $B_\om\subset c_0(\Nat)$, defined~by
\begin{equation}\label{eq:Fein-alg}
 B_\om= \{ f\in c_0(\Nat) \st \abs{f_1}\om(1)+ \sum_{j\geq 2} \abs{f_{j+1}-f_j}\om(j) < \infty \}.
\end{equation}
equipped with the obvious norm.
The unitizations of the algebras $B_\om$ are isomorphic to examples studied in \cite{Fein_alg}. Note that in the cases where $\om$ is bounded, $\lom(\Nmin)$ is isomorphic as a Banach algebra to $\ell^1(\Nmin)$, and $B_\om$ will consist of those $c_0$-sequences which have bounded variation.
For more details, see \cite[\S3]{DalesLoy_diss}.
\end{eg}

\begin{rem}
In the literature, terminology and notation for these examples has varied.
Strictly speaking, Feinstein's construction in \cite{Fein_alg} is more general: he defines, for any given sequence $\al=(\al_n)_{n\geq 1}$ of strictly positive real numbers, a commutative unital Banach algebra $A_\al\subset C(\Nat\cup\{\infty\})$, and in his notation our $B_\al$ coincides with $A_\al \cap c_0$. Note, however, that in some later papers the algebra $A_\al$ is defined to be what we have denoted by $B_\al$.
\end{rem}

For general semilattices, a systematic approach to the character space is via the language of \dt{filters}. First we introduce some notation that will be used later in Section~\ref{s:AMNM-examples}.
Given a semilattice $S$ and a subset $E\subseteq S$, we define the following sets:
\begin{equation}\label{eq:n-gen}
\gen[n]{E}= \{ x_1\dotsb x_n \st x_1,\dots, x_n \in E \} \quad(n\geq 1),
\end{equation}
and $\gen{E} = \bigcup_{n\geq 1} \gen[n]{E}$.
Note that $\gen{E}$ is the sub-semilattice of $S$ generated by $E$.

\begin{dfn}
Let $S$ be a semilattice and let $F\subseteq S$. We say that $F$ is a \dt{filter in $S$} if it satisfies the following three properties:
it is non-empty;
it is closed under multiplication (i.e.~$xy\in F$ whenever $x,y\in F$);
and it is upwards closed in $S$ (i.e.~whenever $x\in S$ and $y\in F$ with $xy\in F$, then $x\in F$).
\end{dfn}

Given a non-empty subset $E$ in a semilattice $S$, there exists a smallest filter in $S$ which contains $E$. This can be described concretely: it is the set
\begin{equation}
F=\bigcup_{y\in \gen{E}} \{ x \in S \st x\succeq y\}.
\end{equation}

The following lemma assembles some basic results, which are well known, at least implicitly, for multiplicative linear functions on semilattice algebras. (See e.g.~\cite[11.1.1]{HewZuck_TAMS56}.) Since the proofs are no harder for the weighted versions, we leave them to the reader.

\begin{lem}\label{l:which-sets}
Let $S$ be a semilattice and $F\subset S$. The following are equivalent:
\begin{YCnum}
\item $F$ is a filter in $S$;
\item the indicator function of $F$ is a character on $\lom(S)$, for every submultiplicative weight~$\om$;
\item there exist a submultiplicative weight $\om$ such that the indicator function of $F$ is a character on~$\lom(S)$.
\end{YCnum}
Moreover, if $\phi:S\to\Cplx$ is multiplicative and non-zero, and $\phi=1$ on some subset $E\subseteq S$, then $\phi=1$ on the filter generated by~$E$.
\end{lem}

\end{section}

\begin{section}{The AMNM property for weighted semilattice algebras}

\subsection{The unweighted case, as a guide}

\begin{thm}\label{t:AMNM-no-weight}
Let $S$ be a semilattice. Then $\ell^1(S)$ is AMNM.
\end{thm}

This result, or equivalent reformulations, may have been implicitly known to previous authors. We shall give a complete proof, since it allows us to introduce some basic results and techniques which will be instructive for later arguments.

Define $f: [0,1/4] \to [0,1/2]$ by
\begin{equation}\label{eq:key-function}
f(t) = \frac{1}{2}\left(1-\sqrt{1-4t}\right).
\end{equation}
and define $\rho(t)=t^{-1}f(t)$ for $0<t\leq 1/4$. The following properties are straightforward to verify:
\begin{itemize}
\item $f$ is convex and monotone increasing, and $f(0)=0$;
\item $\rho$ is monotone increasing on $(0,1/4]$, and $\lim_{t\searrow 0} \rho(t) = 1$.
\end{itemize}

\begin{notn}
Given $r\geq 0$ and $w\in\Cplx$, denote the closed disc of radius $r$ and centre $w$ by~$\clD_w(r)$. If $r>0$, denote the open disc of radius $r$ and centre $w$ by~$\oD_w(r)$.
\end{notn}

\begin{lem}\label{l:in-C}
Let $z\in \Cplx$ and $\veps\in [0,1/4)$. If $z^2-z\in\clD_0(\veps)$, then $\dist_{\Cplx}(z,\{0,1\}) \leq \rho(\veps)\veps$.
\end{lem}

\begin{proof}
Put $w=z-z^2$. Then $\left(z-\frac{1}{2}\right)^2 = \frac{1}{4}(1-4w)$, and so
$z- \frac{1}{2} = \pm \frac{1}{2}\sqrt{1-4w}$ (taking the branch of the square root function for which $\sqrt{1-4w}\to 1$ as $w\to 0$).
It follows that
\[ \dist_{\Cplx}(z,\{0,1\}) \leq \Abs{\frac{1}{2}-\frac{1}{2}\sqrt{1-4w} } .\]
The Taylor expansion (about $0$) of the function $w\mapsto 1-\sqrt{1-4w}$ has non-negative coefficients. Hence $\dist_{\Cplx}(z,\{0,1\})\leq f(\abs{w})$, and the rest follows from our earlier observations.
\end{proof}

Theorem~\ref{t:AMNM-no-weight} follows immediately from the following technical result.
\begin{prop}\label{p:scalars}
Let $S$ be a semilattice, and let $\psi:S \to \Cplx$ satisfy $\diff(\psi) < 1/5$.
Let
\[ S_1 = \psi^{-1}\left(\oD_1(7/25)\right) = \{ e \in S \st \abs{\psi(e)-1} < 7/25 \}, \]
and let $\chi$ be the indicator function of $S_1$. Then $\chi$ is a multiplicative function $S\to \Cplx$, satisfying
\[
\norm{\psi-\chi}_\infty \leq \frac{7}{5} \diff(\psi).
\]
\end{prop}

\begin{proof}
We first note that
\begin{equation}\label{eq:rho(1/5)}
\rho\left(\frac{1}{5}\right) = \frac{5}{2} \left( 1- \sqrt{\frac{1}{5}}\right) = \frac{1}{2}(5-\sqrt{5}) < \frac{1}{2}(5-2.2) = \frac{7}{5}\,.
\end{equation}
Hence, by Lemma~\ref{l:in-C}, we have
\[ \sup_{e\in S} \min( \abs{\psi(e)-1}, \abs{\psi((e)}) \leq \sup_{e\in S} \frac{7}{5}\abs{\psi(e)^2-\psi(e)} =\frac{7}{5}\diff(\psi) < \frac{7}{25} .\]
Set $S_0= \psi^{-1}(\oD_0(7/25))$. Then $S=S_0\sqcup S_1$, so that the previous inequality immediately yields $\norm{\psi-\chi}_\infty \leq (7/5)\norm{\psi^2-\psi}_{\infty} = (7/5)\diff(\psi)$ (see~\eqref{eq:defect-formula}).

To prove $\chi$ is multiplicative, it suffices by Lemma~\ref{l:which-sets} to show that $S_1$ is either empty or a filter. Suppose that $S_1$ is non-empty.
If $e,f\in S_1$, then
$\abs{\psi(e)\psi(f)} > (18/25)^2 > 1/2$.
Hence
\[
\abs{\psi(ef)} > \frac{1}{2} - \diff(\psi) > \frac{1}{2} - \frac{1}{5} > \frac{7}{25}, \]
forcing $ef\in S\setminus S_0 = S_1$.
If $e\succeq f\in S$ and $f\in S_1$, then since $\abs{\psi(f)}^{-1} < 25/18$, we have
\[ \abs{\psi(e)-1} \leq \frac{25}{18}\abs{\psi(e)\psi(f)-\psi(f)} \leq \frac{25}{18}\diff(\psi) < \frac{5}{18} < \frac{7}{25} , \]
forcing $e\in S_1$. Thus $S_1$ is a filter, and the proof is complete.
\end{proof}

\subsection{A weighted semilattice which is not AMNM}

\begin{thm}\label{t:not-AMNM}
There exists a locally finite semilattice $T$ and a submultiplicative weight $\om$ on $T$ such that $\lom(T)$ is not AMNM.
\end{thm}

The counter-example $T$ will be built out of copies of free semilattices.
Given a set $S$, let $\free{S}$ denote the \dt{free semilattice} generated by~$S$; this can be identified with the set of all non-empty finite subsets of $S$, with semigroup product given by union. For our purposes it is more convenient to regard elements of $\free{S}$ as reduced words in the generators.
There is a natural length function $\gamma_S: \free{S} \to \Nat$, where $\gamma_S(x)$ is the minimum number of elements in $S$ needed to generate $x$.
A little thought shows that $\gm_S(xy)\leq \min(|S|, \gm_S(x)+\gm_S(y))$ for all $x,y\in \free{S}$.

If $S$ is a finite set, then $\free{S}$ has a zero (i.e. minimum) element, namely the product of all elements of $S$; we denote this element by~$\theta_S$.
 Note that $\gamma_S(\theta_S)=|S|$. 

\begin{prop}\label{p:building-block}
Let $S$ be a finite set of cardinality $\geq 2$. 
Fix a constant $C>1$, and define $\om_S:\free{S}\to [1,\infty)$ by
\[ \om_S(e) = \left\{ \begin{aligned}
        C^{\gm_S(e)} & \quad\text{if $e\neq \theta_S$} \\
        C       & \quad\text{if $e=\theta_S$}
\end{aligned} \right. \]
Then $\om_S$ is a submultiplicative weight on $\free{S}$. Moreover, if we define $\psi: \free{S} \to \{0,1\}$~by
\[ \psi(e) = \left\{ \begin{aligned}
        1 & \quad\text{if $e\neq\theta_S$,} \\
        0 & \quad\text{if $e=\theta_S$,}
        \end{aligned}\right. \]
then:
\begin{YCnum}
\item as an element of $\ell^1_{\om_S}(\free{S})^*$, $\diff(\psi)\leq C^{-|S|}$\/;
\item if $\phi:\free{S} \to \{0,1\}$ is multiplicative, then
\[ \sup_{e\in \free{S}} \om_S(e)^{-1} \abs{ \psi(e) - \phi(e) } \geq C^{-1}\,. \]
\end{YCnum}
\end{prop}

\begin{proof}
Since $\gm_S$ is subadditive and $\om_S(e)\leq C^{\gm_S(e)}$ for all $e\in \free{S}$, it is clear that $\om_S$ is submultiplicative.

To prove~(i), let $e,f\in\free{S}$. If $ef\neq\theta_S$ then
$\psi(e)\psi(f) = 1 = \psi(ef)$. If $e=\theta_S$ or $f=\theta_S$ then
$\psi(e)\psi(f) = 0 = \psi(ef)$.
The only remaining cases are those where $e,f\neq\theta_S$ while $ef=\theta_S$. In such cases, we have $\gm_S(e)+\gm_S(f)\geq |S|$, and so
\[
\abs{\psi(e)\psi(f)-\psi(ef)} = 1 \leq \delta C^{-|S|} \om(e)\om(f).
\]
Thus in all cases, $\abs{\psi(e)\psi(f)-\psi(ef)} \leq C^{-|S|} \om(e)\om(f)$.

To prove (ii), let $\phi: \free{S} \to \{0,1\}$ be multiplicative.
If $\phi(\theta_S)=1$ then
\[ \om(\theta_S)^{-1}\abs{\phi(\theta_S)-\psi(\theta_S)} = C^{-1} \,.\]
If not, then $\phi(\theta_S)=0$. Therefore, since $\phi$ is multiplicative and $S$ generates $\free{S}$, there exists $s_0\in S$ with $\phi(s_0)=0$, and so
\[ \om(s_0)^{-1}\abs{\phi(s_0)-\psi(s_0)} = C^{-1}\,.\]
\end{proof}

If $(F_i)_{i\in\Ind}$ is a family of semilattices, 
consider the set $\{\theta\}\sqcup \coprod_{i\in \Ind} F_i$, where $\theta$ is a formal symbol. This can be made into a semilattice if we define the product as follows: $\theta$ is an absorbing zero element; the product of two elements in $F_i$ is their usual product; the product of elements in $F_i$ and $F_j$ is $\theta$ whenever $i\neq j$. We call this semilattice the \dt{orthogonal direct sum} of the family $(F_i)_{i\in\Ind}$.

\begin{proof}[Proof of Theorem~\ref{t:not-AMNM}]
Let $(F_n)_{n\geq 1}$ be a sequence of finite sets with $\abs{F_n}\nearrow \infty$.
Define $T$ to be the orthogonal direct sum of the family $(\free{F_n})_{n\in\Nat}$, with $\theta$ being the zero element of $T$.
To ease notation, denote the zero element of $\free{F_n}$ by $\theta_n$, and denote the length function of $\free{F_n}$ by $\gm_n$.

Fix a constant $C>1$, and define $\om:T \to [1,\infty)$ by
\[ 
\om(\theta) \defeq 1 \quad, \quad\om(\theta_n) \defeq C \;\text{for each $n\in\Nat$} \quad,\quad \om(e) \defeq C^{\gm_n(e)}  \quad\text{if $e\in \free{F_n}\setminus\{\theta_n\}$.}
\]
Define $\psi_n: T \to \{0,1\}$ by
\[ \psi_n(e) = \left\{ \begin{aligned}
        1 & \quad\text{ if $E\in S_n \setminus \{ \theta_n \}$,} \\
        0 & \quad\text{ otherwise.}
        \end{aligned}\right. \]

If $\phi:T \to \{0,1\}$ is multiplicative, then by part (ii) of Proposition~\ref{p:building-block},
\[ \norm{\psi_n-\phi} \geq \sup \{ \om(e)^{-1} \abs{\psi_n(e)-\phi(e)} \st e\in \free{F_n} \} \geq C^{-1}\,. \]
Hence,
\[ \dist_{\lom(T)^*}(\psi_n,\Mult(\lom(T),\Cplx))\geq C^{-1} \quad\text{for all $n$.} \]
On the other hand, by Proposition~\ref{p:building-block}(i),
\[ \diff_\om(\psi_n) = \sup\left\{ \frac{ \abs{\psi_n(e)\psi_n(f)-\psi_n(ef)}}{\om(e)\om(f)} \st e,f\in\free{F_n} \right\}
\leq C^{-|F_n|} \to 0. \]
Thus $\lom(T)$ is not AMNM.
\end{proof}

The example $(T,\om)$ was found while trying to prove that all weighted semilattice algebras are AMNM, and realizing that the attempted proof only worked when one could verify a certain technical condition on a given weighted semilattice $(S,\om)$. This condition, and the proof that it suffices to ensure $\lom(S)$ is AMNM, will be our next topic.

\subsection{Weighted semilattices which are AMNM}\label{s:AMNM-examples}
The following lemma is a substitute for Lemma~\ref{l:in-C}. It is less informative in the case of $\ell^1(S)$, but is more convenient in weighted cases.

\begin{lem}\label{l:easy}
Let $A$ be a Banach algebra, $e$ an idempotent in $A$, and $\psi\in A^*$. Then
\begin{equation}\label{eq:cheap-bound}
\min( \abs{\psi(e)}, \abs{1-\psi(e)}) \leq \diff(\psi)^{1/2} \norm{e}.
\end{equation}
\end{lem}

\begin{proof}
If $e=0$ this is trivial. If $e\neq 0$, observe that
\[ \diff(\psi) \geq \frac{\abs{\psi(e)-\psi(e)^2}}{\norm{e}^2} =
 \frac{\abs{\psi(e)}}{\norm{e}} \frac{\abs{1-\psi(e)}}{\norm{e}} \geq
\left( \frac{\min( \abs{\psi(e)}, \abs{1-\psi(e)}) }{\norm{e}}\right)^2  \,,\]
and \eqref{eq:cheap-bound} follows.
\end{proof}

By \eqref{eq:cheap-bound} and the definition of the norm on~$\lom(S)$, we see that if $\psi:S \to \Cplx$ is $\om$-bounded with $\diff_{\om}(\psi)\leq\delta$, then there exists an $\om$-bounded function $\phi:S \to \{0,1\}$ such that $\norm{\psi-\phi}_{\infty,\om^{-1}}\leq \delta^{1/2}$.
Since
\[ \begin{aligned}
\abs{\phi(s)\phi(t)-\phi(st)}
&  \leq \left\{ \begin{gathered}
    \abs{\phi(s)}\abs{\phi(t)-\psi(t)} + \abs{\phi(s)-\psi(s)}\abs{\psi(t)} \\
    + \abs{\psi(s)\psi(t)-\psi(st)} + \abs{\psi(st)-\phi(st)}
  \end{gathered} \right. \\
& \leq 1\cdot \delta^{1/2}\om(t) + \delta^{1/2}\om(s) \cdot (1+\delta^{1/2}\om(t))
   + \delta\om(s)\om(t) + \delta^{1/2}\om(st) \\
&  =   \delta^{1/2}\om(t) + \delta^{1/2}\om(s) + 2\delta\om(s)\om(t)  + \delta^{1/2}\om(st) \;,\\
 \end{aligned} \]
we have $\diff_\om(\phi) \leq 3\delta^{1/2}+2\delta$.
From this, routine arguments (which we omit) yield the following necessary and sufficient condition for $\lom(S)$ to be AMNM.

\begin{cor}\label{c:reduced-problem}
Let $(S,\om)$ be a weighted semilattice. \TFAE:
\begin{YCnum}
\item $\lom(S)$ is AMNM;
\item For any $\veps>0$, there exists $\delta>0$, such that whenever $\phi: S \to \{0,1\}$ satisfies
\[ \sup_{e,f\in S} \frac{\abs{\phi(ef)-\phi(e)\phi(f)}}{ \om(e)\om(f)} \leq \delta\;, \]
then there exists a subset $F\subseteq S$, either empty or a filter, such that
\[ \sup_{e\in S} \frac{\abs{ \chi_F(e)-\phi(e)}}{\om(e)} \leq \veps. \]
\end{YCnum}
\end{cor}

It turns out that a certain natural, structural condition on a semilattice $S$ will imply $\lom(S)$ is AMNM for \emph{any} choice of weight function $\om$.
This condition was rediscovered by the author in the course of the present investigations; it was subsequently pointed out (personal communication, see \cite{MO100971}) that it was already known by a standard name in the lattice-theoretic literature.

\begin{dfn}[The breadth of a semilattice]\label{d:UFG}
Let $E$ be a subset of a semilattice $S$. We define
\[ b_{\rm loc}(E) = \inf\{ n\in\Nat \st \gen[n]{E}=\gen{E}\},  \]
with the usual convention that the infimum of the empty set is $+\infty$.
The \dt{breadth of $S$} is defined to be $\sup_{E\subseteq S} b_{\rm loc}(E)$, and is denoted by $b(S)$.
\end{dfn}

Here, the sets $\gen[n]{E}$ and $\gen{E}$ are as defined in \eqref{eq:n-gen}.
Note that having finite breadth is not the same as being finitely generated as a semigroup.
Indeed, finitely generated semilattices are finite.

The following examples of semilattices with finite breadth are presumably well known to specialists, but since we were unable to find explicit references in the literature, details of the proofs are included. (For some discussion, see e.g.~\cite[\S IV.10]{Birkhoff_3rdcorr}.)

\begin{eg}
Let $S$ be the free semilattice on $n$ generators. Then $b(S)=n$.

\begin{proof}
Since $S$ has height $n$, it will follow from Example~\ref{eg:finite-height} below that
 $b(S)\leq n$. On the other hand, if $F$ denotes the set of generators of $S$, then $\gen[n]{F}=S$ while $\gen[n-1]{F}$ does not contain the minimal element of $S$, implying that $b(S)\geq b_{\rm loc}(F) > n-1$.
\end{proof}
\end{eg}

\begin{eg}
If $S$ has width $\leq n$ as a partially ordered set (i.e.~there exist $n$ chains in $S$ whose union is all of $S$) then $b(S)\leq n$.

\begin{proof}
Fix chains $C_1, \dots, C_n$ in $S$ such that $S=\bigcup_{i=1}^n C_i$. Let $E\subseteq S$ and $x\in \gen{E}$.
For some integer $k$ there exists $y_1,\dots, y_k\in E$ such that $x=y_1\dotsb y_k$. 
Partition $\{1,\dots,k\}$ into disjoint, non-empty subsets $J(1),\dots, J(m)$, where $m\leq n$, such that $y_r \in C_i$ for all $r\in J(i)$. For each $i$, the set $\{ y_r \st r \in J(i)\}$ is totally ordered (as a subset of $S$) and so has a least element, say $y_{s(i)}$. Then $\prod_{r\in J(i)} y_r = y_{s(i)}$, so
\[ x = y_{s(1)} \dots y_{s(m)} \in \gen[m]{E} \subseteq \gen[n]{E} \] 
and thus $b_{\rm loc}(E)\leq n$.)
\end{proof}
\end{eg}

\begin{eg}\label{eg:finite-height}
Let $n\geq 2$.
If $S$ has height $\leq n$ (i.e.~each chain in $S$ has cardinality $\leq n$) then $b(S)\leq n$.

\begin{proof}
Let $E\subseteq S$ and $x\in\gen{E}$.
For some integer $k$ there exists $y_1,\dots, y_k\in E$ such that $x=y_1\dotsb y_k$. Clearly
$y_1 \succeq y_1y_2 \succeq \dots \succeq y_1y_2\dotsb y_k$. Let
\[ J=\{ r \in \{2,\dots,k\} \st y_1\dotsb y_{r-1} \neq y_1\dotsb y_r \}, \]
and enumerate its elements in increasing order as $j_1<\dots <j_m$.
Since chains in $S$ have cardinality at most $n$, we have $m\leq n+1$, and thus $b_{\rm loc}(E)\leq n$.)
\end{proof}
\end{eg}

We need one more technical definition. The following concept is a slight improvement, suggested by the referee, of the author's original condition. However, the author takes the blame for the non-standard terminology.

\begin{dfn}\label{d:ref-improved-condition}
Let $(S,\om)$ be a weighted semilattice. For each $K>0$, let
\[ W_K=\{ x\in S \colon \om(x)\leq K\}.\]
We say that $(S,\om)$ is \dt{flighty} if, for each $K>0$,
\[ \sup \{ \om(y) \st y \in \gen{W_K} \} < \infty\,. \]
\end{dfn}
This condition is somewhat artificial, and is set up to make the proof of Theorem~\ref{t:is-AMNM} work. Let us first consider some examples.

\begin{eg}\label{eg:cases}\
\begin{YCnum}
\item\label{li:one} If $\sup_{x\in S} \om(x) < \infty$, then $(S,\om)$ is flighty.
\item\label{li:two} If $S$ has finite breadth $n$, then $(S,\om)$ is flighty for any weight $\om$; for given $K>0$, we have $\om(y)\leq K^n$ for all $y\in \gen{W_K}$.
\item Let $(T,\om)$ be the weighted semilattice constructed in the proof of Theorem~\ref{t:not-AMNM}. This example is not flighty: since (if $C$ is the constant used to define the weight in that example) we can for each $n$ find $x_1,\dots, x_n \in W_C$ such that $\om(x_1\dotsb x_n) = C^n$.
\end{YCnum} 
\end{eg}

We can now state the main result of this section. The proof we give incorporates a small simplification suggested by the referee, in line with his or her suggested modification of Definition~\ref{d:ref-improved-condition}.

\begin{thm}\label{t:is-AMNM}
Let $(S,\om)$ be a weighted semilattice. If $(S,\om)$ is flighty, then $\lom(S)$ is AMNM.
\end{thm}

\begin{proof}
We will use Corollary~\ref{c:reduced-problem}.
Fix $\veps>0$, and let
\[
\Sfix \defeq W_{2/\veps} \equiv \{x \in S \st \om(x) \leq 2/\veps \}, \]
which we think of this as the set where $\om$ is relatively small.

By our assumption on $(S,\om)$, there exists some constant $C(\veps) > 0$ such that
$\om(y) \leq C(\veps)$ for all $y\in \gen{\Sfix}$.
Choose $\delta>0 $ such that $2\delta C(\veps)/ \veps < 1$.

Now suppose $\psi: S\to \{0,1\}$ satisfies $\diff_\om(\psi)\leq \delta$.
Let 
\[ E \defeq \{ y\in \Sfix \st \psi(y)=1 \} \]
and let
\[ F \defeq \bigcup_{y\in \gen{E}} \{x \in S \st x\succeq y\}\,. \]
Recall: if $E$ is empty then so is $F$; otherwise, $F$ is the filter generated by $E$ in $S$.

To motivate the next step, suppose $\phi:S\to\Cplx$ is multiplicative and is close in norm to~$\psi$. Then (provided $\delta$ is sufficiently small) we must have $\phi=1$ on $E$; and so, as remarked in Lemma~\ref{l:which-sets}, $\phi$ must be $1$ on~$F$. Therefore, $\psi$ must also be $1$ on $F\cap\Sfix$, i.e.~we must have $E=F\cap\Sfix$.

So, our next step is to verify that $E=F\cap\Sfix$. It suffices to show that $F\cap\Sfix\subseteq E$, and the  proof of this goes via the following claim:

\para{Claim.} Let $k\in\Nat$. Then
$\psi(y)=1$ for all $y\in \gen[k]{E}$.

\para{Proof of claim.}
Induction on $k$. If $k=1$ there is nothing to prove. Suppose that the claim holds for $k=m-1$ where $2\leq m$, and let $x_1,\dots, x_m\in E$. Put $y'=x_1\dotsb x_{m-1}$; then by the inductive hypothesis, the definition of the constant $C(\veps)$, and our choice of $\delta$, we have
\[ \begin{aligned}
\abs{ \psi(y' x_m) - 1 }
 & = \abs{\psi(y'x_m) - \psi(y')\psi(x_m)} \\
 & \leq \delta \om(y')\om(x_m)  \\
 & \leq \delta C(\veps)\om(x_m)  \\
 & \leq \delta C(\veps) 2/\veps  < 1.
\end{aligned} \]
Since $\psi$ is $0$-$1$-valued, this forces $\psi(y'x_m)=1$, so that the claim holds for $k=m$.
\medskip

Now let $z\in F\cap\Sfix$. Since $z\in F$, there exists $y\in \gen{E}$ such that $zy=y$.  By our claim, $\psi(y)=1$. Since $z\in\Sfix$ and $E\subseteq \Sfix$, this implies
\[ \begin{aligned}
\abs{\psi(z)-1}
 & = \abs{\psi(y)\psi(z)-\psi(yz)} \\
 & \leq \delta\om(y)\om(z) \\
 & \leq \delta C(\veps) \om(z) \\
 & \leq \delta C(\veps) 2/\veps  < 1.
\end{aligned} \]
As before, this forces $\psi(z)=1$, so that $z\in E$.

Thus $F\cap\Sfix = E$, as required. We therefore have $S= E \cup (\Sfix\setminus F) \cup (S\setminus\Sfix)$. Now observe that:
\begin{itemize}
\item when $x\in E$, we have $\abs{\chi_F(x)-\psi(x)} = 0$;
\item when $x\in \Sfix \setminus F\subseteq \Sfix\setminus E$, we have $\chi_F(x)=0=\psi(x)$, so that $\abs{\chi_F(x) - \psi(x)} = 0$;
\item when $x\in S\setminus \Sfix$, we have $\om(x) \geq 2/\veps$, so that $\abs{\chi_F(x)-\psi(x)}\leq 2 \leq \veps\om(x)$.
\end{itemize}
Putting these cases together, we see that $\chi_F$ is multiplicative and satisfies\hfill\newline $\sup_{x\in S} \om(x)^{-1}\abs{\chi_F(x)-\psi(x)} \leq \veps$. In view of Corollary~\ref{c:reduced-problem}, this shows that $\lom(S)$ is AMNM.
\end{proof}

As a special case of Theorem~\ref{t:is-AMNM}, we get another proof that $\ell^1(S)$ is AMNM for every semilattice $S$ (see Example~\ref{eg:cases}\ref{li:one}). More interestingly, we can deduce that if $S$ is a semilattice with either finite width or finite height, then $\lom(S)$ is AMNM for every submultiplicative weight~$\om$ (see Example~\ref{eg:cases}\ref{li:two}). In particular, for any weight function $\om:\Nat\to [1,\infty)$, the algebra $B_\om$ from Example~\ref{eg:fein-alg} is AMNM, as it is isomorphic to $\lom(\Nmin)$.

\end{section}

\begin{section}{More general range algebras}
For which Banach algebras $B$ and weighted semilattices $(S,\om)$ is $(\lom(S),B))$ an AMNM pair? Since we do not have a complete answer in the case $B=\Cplx$, we can expect only partial results in the more general case. A natural place to start is with those semilattices covered by Theorem~\ref{t:is-AMNM}, in particular with $\Nmin$. Recall that by this theorem, $\lom(\Nmin)$ is AMNM for any weight $\om$.
In contrast, we will now show that whenever $\om$ is a non-trivial (=unbounded) weight on $\Nmin$, the pair $(\lom(\Nmin),\sT_2)$ is not AMNM, and that the pair $(\lom(\Nmin),\Mat{2})$ fails to be AMNM for most choices of weight~$\omega$. Here, $\Mat{2}$ is the algebra of $2\times 2$ complex matrices, and $\sT_2$ is the subalgebra
\[ \left\{ \twomat{a}{b}{0}{a} \st a,b\in\Cplx\right\}\iso \Cplx[x]/(x^2). \]

\begin{rem}
These are not arbitrary test cases: $\Mat{2}$ is the smallest semisimple algebra that is noncommutative, while $\sT_2$ is the smallest commutative, unital algebra that is not semisimple. (In the context of commutative algebra and deformation theory, it is also known as the algebra of \dt{dual numbers over $\Cplx$}, since it formalizes the notion of an infinitesimal element vanishing to 2nd order.)
Additional motivation comes from results of R. A. J. Howey, who showed that $C^k[0,1]$ is AMNM ($k\geq 1)$ (\cite{Howey_JLMS}), while observing -- in slightly different notation -- that the pair $(C^k[0,1], \sT_2)$ is not AMNM (\cite[Corollary 4.2.5]{Howey_PhD}; see also Remark~\ref{r:approx-der} below).
\end{rem}

It does not matter which norm we put on $\sT_2$ or $\Mat{2}$. To be definite, we give $\Mat{2}$ its natural norm (the $\Cst$-algebra norm), but equip $\sT_2$ with the norm
\[ \Norm{ \twomat{a}{b}{0}{a} }_{\sT_2} = \abs{a}+\abs{b} . \]
We now consider AMNM pair problems for $\lom(\Nmin)$ when the range is $\sT_2$ or $\Mat{2}$.
 If the weight function is bounded then $\lom(\Nmin)$ is isomorphic as a Banach algebra to $\ell^1(\Nmin)$, and so we may restrict attention to the cases where the weight is either trivial (i.e.~identically $1$) or unbounded.

\begin{thm}\label{t:wt-Nmin-T2}
Let $\om$ be an unbounded weight function on $\Nmin$.
Then $(\lom(\Nmin), \sT_2)$ is not an AMNM pair.
\end{thm}

\begin{proof}
Throughout this proof, let $A$ denote the algebra $\lom(\Nmin)$.
For each $m\in\Nat$ let $\chi_m$ be the character
  \[ \chi_m(k) = \left\{\begin{aligned}
  1 & \quad\text{if $k\geq m$,} \\
  0 & \quad\text{if $k < m$,}
  \end{aligned} \right. \]
  and let $\delta_m$ be the Dirac point mass at~$m$.
Set
 \[ \theta_m =\twomat{\chi_m}{\om(m)\delta_m}{0}{\chi_m} : \Nat \to \sT_2, \]
noting that the range of $\theta_m$ consists of commuting elements.

For $1\leq j\leq k$, we have
 \[ \begin{aligned}
\theta_m(j)\theta_m(k) - \theta_m(j)
   & = \twomat{0}{\chi_m(j)\om(m)\delta_m(k)+\om(m)\delta_m(j)(\chi_m(k)-1)}{0}{0} \\
   & = \twomat{0}{\chi_m(j)\om(m)\delta_m(k)}{0}{0} .
   \end{aligned}\]
Hence by symmetry, for general $j,k\in\Nmin$ we have
\[ \begin{aligned}
  \theta_m(j)\theta_m(k) - \theta_m(j\wedge k)
 & = \twomat{0}{\chi_m(j\wedge k)\om(m)\delta_m(j\vee k)}{0}{0} \\
\end{aligned}
\]
which vanishes unless $j=k=m$, in which case
 \[  \theta_m(m)\theta_m(m) - \theta_m(m) = \twomat{0}{\om(m)}{0}{0}.\]
Thus
\begin{equation}\label{eq:chocolate}
 \diff(\theta_m) = \sup_{j, k\in\Nat} \frac{\norm{\theta(j)\theta(k)-\theta(j\wedge k)}}{\om(j)\om(k)}
 = \frac{\om(m)}{\om(m)\om(m)} = \om(m)^{-1}.
\end{equation}

Suppose $\Phi: \Nmin \to \sT_2$ is multiplicative. Since $\sT_2$ has no non-trivial idempotents, this forces $\phi(n)_{12}=0$ for all $n$.
Hence
\[ \norm{\Phi(m)-\theta_m(m)} \geq \abs{\Phi(m)_{12}-\theta_m(m)_{12}}
 \geq \om(m), \]
so that
\[ \dist_{\Lin(A,\sT_2)} (\theta_m,\Mult(A,\sT_2)) \geq \sup_{x\in\Nat} \frac{\norm{\Phi(x)-\theta_m(x)}}{\om(x)} \geq 1 \quad\text{for all $m$.} \]
Since $\liminf_n \om(n)^{-1}=0$, it follows that $(A,\sT_2)$ is not an AMNM pair.
\end{proof}

\begin{rem}\label{r:approx-der}
A (bounded) multiplicative function $\Psi$ from a Banach algebra $A$ to $\sT_2$ is easily seen to be of the form
\[ \Psi(a) = \twomat{\varphi(a)}{D(a)}{0}{\varphi(a)},\]
where $\varphi\in\Mult(A,\Cplx)$ and $D:A \to \Cplx$ is a (bounded) derivation with respect to the bimodule action of $A$ on $\Cplx$ via $\phi$. In the case where $A$ has a dense subspace consisting of commuting idempotents, the only such bounded derivation is $0$; and the failure of $(\lom(\Nmin),\sT_2)$ to be AMNM can be interpreted as saying that there are ``approximate point derivations'' of large norm on $\lom(\Nmin)$. This perspective (motivated by cohomological questions about the ``Fein\-stein algebras'') was the original approach used to construct the counter-example seen in proving Theorem~\ref{t:wt-Nmin-T2}. With hindsight, a similar idea can be seen behind Howey's result that $(C^k[0,1],\sT_2)$ is not an AMNM pair.
\end{rem}

\medskip
In general, approximately multiplicative maps into an algebra $B$ can be far from multiplicative maps into $B$, yet be close to multiplicative maps into $C$ for some containing algebra $C\supset B$. This is illustrated by the following example, provided by the referee.

\begin{eg}[Referee's example]
Let $\chi_m$ and $\theta_m$ be as in the proof of Theorem~\ref{t:wt-Nmin-T2}. We saw in that proof that $\diff(\theta_m)=\om(m)^{-1}$ and
\[ \dist_{\Lin(A,\sT_2)} (\theta_m,\Mult(A,\sT_2)) \geq  1. \]
Now define $\phi_m: \Nmin\to \Mat{2}$ by
\[ \phi_m = \twomat{\chi_m}{\om(m)\delta_m}{0}{\chi_{m+1}} \]
A case-by-case analysis confirms that $\phi_m$ is multiplicative, the key point being that $\phi_m(k)=0$ for $k<m$ and $\phi_m(k)=I$ for $k>m$. Moreover,
\[ \theta_m - \chi_m = \twomat{0}{0}{0}{\delta_m} \]
so that $\norm{\theta_m-\chi_m} = \om(m)^{-1} = \diff(\theta_m)$.
\end{eg}

One reason why things are trickier when the range algebra is $\Mat{2}$ is that the lattice of idempotents is bigger. Nevertheless, we will be able to use the following modest observation:
{\it if $V$ is a finite-dimensional vector space, and $P,Q$ are commuting idempotents in $\Lin(V)$ with the same rank, then $P=Q$.}
While this observation is easily proved by geometric considerations, we outline an alternative approach: since $P$ and $Q$ are commuting idempotents, $P-Q$ is idempotent; and since the rank of an idempotent equals its trace, it follows that
\[ \operatorname{rank}(P-Q) = \tr(P-Q)=\tr(P)-\tr(Q) = \operatorname{rank}(P)-\operatorname{rank}(Q) = 0 \]
so that $P-Q=0$, as required.

\begin{rem}\label{r:harbinger}
We mention this approach since an ``approximate version'' will be used in Section~\ref{s:AMNM-M2} when dealing with $2\times 2$ matrices that are ``approximately idempotent'': while the rank of a matrix behaves badly under small-norm perturbations, its trace behaves much better.
\end{rem}

\begin{lem}\label{l:two-element-obstruction}\
\begin{YCnum}
\item Let $a,b \geq 1$ and let $A=\twomat{1}{-a}{0}{0}$, $B=\twomat{1}{b}{0}{0}$.
If $P, Q$ are commuting idempotents in $\Mat{2}$, then either $\norm{P-A}\geq a/2$ or $\norm{Q-B}\geq b/2$.
\item Let $d \geq 1$ and let $C=\twomat{1}{d}{0}{0}$. If $P, Q$ are commuting idempotents in $\Mat{2}$, then either $\norm{P-2C}\geq d/2$ or $\norm{Q-C}\geq d/4$.
\end{YCnum}
\end{lem}

\begin{proof}
We prove both (i) and (ii) by contradiction. Suppose $P$ and $Q$ are commuting idempotents that satisfy $\norm{P-A} < a/2$ and $\norm{Q-B} < b/2$. 
Since $\norm{A}\geq a$ and $\norm{A-I}\geq a$, $P\notin\{0,I\}$; similarly, $Q\notin\{0,I\}$. Hence $P$ and $Q$ both have rank $1$. This forces $P=Q$ (see the observation made before Remark~\ref{r:harbinger}), and so
\[ 0< a+b = \norm{A-B} \leq \norm{A-P}+\norm{Q-B} < (a+b)/2 \]
which is a contradiction. Thus (i) is proved.

Similarly, suppose $P$ and $Q$ are commuting idempotents that satisfy $\norm{P-2C} < d/2$ and $\norm{Q-C} < d/4$. Since $\norm{C}\geq d$ and $\norm{C-I}\geq d$, $Q\notin\{0,I\}$; similarly, $P\notin\{0,I\}$. Thus $P$ and $Q$ both have rank $1$, which as before forces $P=Q$. But then
\[ 0 < d = \norm{C} \leq \norm{2C-P}+\norm{Q-C} < \frac{d}{2} + \frac{d}{4} \]
which is a contradiction. Thus~(ii) is proved.
\end{proof}

We can now present our main result for $\lom(\Nmin)$ and $\Mat{2}$, which shows that for many unbounded weights we do not get an AMNM pair. The author is grateful to the referee for spotting errors in an earlier version, which mistakenly claimed that the AMNM property failed for \emph{all} unbounded weights; see also the discussion following Remark~\ref{r:BAC-M2-non-amnm}.

\begin{thm}\label{t:wt-Nmin-M2}
Let $\om$ be a weight function on $\Nmin$ that satisfies
\[ \sup_n \min(\om(n),\om(n+1)) = +\infty .\]
Then $(\lom(\Nmin), \Mat{2})$ is not an AMNM pair.
\end{thm}

\begin{proof}
Let $\delta>0$. We will construct an $\om$-bounded function $\theta: \Nmin\to\Mat{2}$ which has norm $\leq 2$ (with respect to the $\lom$-norm) and satisfies $\diff_\om(\theta)\leq \delta$, yet also satisfies $\norm{\theta-\phi}_{\infty,\om^{-1}}\geq 1/2$ for every multiplicative function $\phi:\Nmin\to\Mat{2}$.

By our assumption on the weight, there exists $n\in\Nat$ such that $\min(\om(n),\om(n+1))\geq 2/\delta$. We define
\[ \begin{aligned}
\theta(j) & = 0 \quad\text{for all $j\leq n-1$;} \\
\theta(n) & = \twomat{1}{-\om(n)}{0}{0} \\
\theta(n+1) & = \twomat{1}{\om(n+1)}{0}{0} \\
\theta(k) & = I \quad\text{for all $k\geq n+2$.}
\end{aligned} \]

Clearly $\norm{\theta}_{\infty.\om^{-1}} \leq 2$. We claim $\theta$ is $\delta$-multiplicative. Clearly $\theta(j)\theta(k)=\theta(j)=\theta(k)\theta(j)$ whenever $j\leq n-1$ or $k\geq n+2$. We also have $\theta(n)^2=\theta(n)$, $\theta(n+1)^2=\theta(n+1)$, $\theta(n+1)\theta(n)=\theta(n)$, and
\[ \theta(n)\theta(n+1)-\theta(n) = \theta(n+1)-\theta(n) = \twomat{0}{\om(n+1)+\om(n)}{0}{0} . \]
Therefore
\[ \diff_\om(\theta) = \frac{\norm{\theta(n)\theta(n+1)-\theta(n)}}{\om(n)\om(n+1)} = \frac{\om(n+1)+\om(n)}{\om(n)\om(n+1)} = \frac{1}{\om(n)}+\frac{1}{\om(n+1)} \leq \delta ,\]
as claimed.

Finally, let $\phi:\Nmin\to \Mat{2}$ be multiplicative. Then $\phi(n)$ and $\phi(n+1)$ are commuting idempotents, so by Lemma~\ref{l:two-element-obstruction}(i),
\[ \text{either } \norm{\theta(n)-\phi(n)}\geq \om(n)/2 \text{ or } \norm{\theta(n+1)-\phi(n+1)}\geq \om(n+1)/2. \]
Therefore $\norm{\theta-\phi}_{\infty,\om^{-1}}\geq 1/2$.
\end{proof}

\begin{rem}\label{r:BAC-M2-non-amnm}
If $\om$ is any weight on $\Nmin$ with $\liminf_n \om(n)<\infty$, then $\lom(\Nmin)$ is known to be approximately amenable (in fact, it is boundedly approximately contractible, by combining the discussion in Example~\ref{eg:fein-alg} with \cite[Corollary 4.5]{GhLZ_genam2}).
Thus, Theorems~\ref{t:wt-Nmin-T2} and \ref{t:wt-Nmin-M2} show we can have approximately amenable Banach algebras $A$ and finite-dimensional Banach algebras $B$ such that $(A,B)$ is not an AMNM pair, in contrast with what happens when $A$ is amenable \cite[Theorem 3.1]{BEJ_AMNM2}.
\end{rem}

Theorem~\ref{t:wt-Nmin-M2} implies a necessary condition on the weight $\om$ for the pair $(\lom(\Nmin),\Mat{2})$ to be AMNM, and one naturally wonders if this necessary condition is sufficient. Put more explicitly: {\it if $\om:\Nmin\to [1,\infty)$ is a weight function satisfying $\sup_n \min(\om(n),\om(n+1))<\infty$, is $(\lom(\Nmin),\Mat{2})$ always an AMNM pair?}

It turns out that the answer to this question is positive. This was discovered after the main work of the present paper, and the current proof is relatively long and unenlightening, while relying on Theorem~\ref{t:slatt-M2-amnm}. Details will therefore be given in forthcoming work, which treats AMNM problems for more general range algebras.
 The main idea is similar to that in the proof of Theorem~\ref{t:is-AMNM}: given $\delta>0$, one partitions $\Nmin$ into a set where the weight is ``large'' and one where it is ``small''; then given a $\delta$-multiplicative function $\theta: \lom(\Nmin)\to\Mat{2}$, one applies an AMNM result for the unweighted case to constrain the values of $\theta$ on the set where the weight is small, and then adjusts $\theta$ on the set where the weight is large.

For sake of completeness, we include the following result, which should be contrasted with Theorem~\ref{t:slatt-M2-amnm}.

\begin{thm}\label{t:wt-Nmin-M2-non-unif}
Let $\om$ be an unbounded weight on $\Nmin$. Then $(\lom(\Nmin),\Mat{2})$ is not a uniformly AMNM pair.
\end{thm}

\begin{proof}
Since uniformly AMNM pairs are {\it a fortiori} AMNM, it suffices by Theorem~\ref{t:wt-Nmin-M2}
 to consider the case where $\om$ is unbounded and $\sup_n \min(\om(n),\om(n+1))<\infty$.

Our argument is very similar to the proof of Theorem~\ref{t:wt-Nmin-M2}. Let $C=\sup_n \min(\om(n),\om(n+1))$, and let $\delta>0$. We will construct an $\om$-bounded function $\theta: \Nmin\to\Mat{2}$ which satisfies $\diff_\om(\theta)\leq \delta$, yet also satisfies $\norm{\theta-\phi}_{\infty,\om^{-1}}\geq 1/2$ for every multiplicative function $\phi:\Nmin\to\Mat{2}$.

The hypotheses on $\omega$ ensure that there exists some $n\in\Nat$ such that $\om(n)\geq \max(6\delta^{-1}, 2C)$ and $\om(n+1)\leq C$.
Define $\theta:\Nmin\to\Mat{2}$ by setting $\theta(j) = 0$ for all $1\leq j \leq n-1$, $\theta(k)=I$ for all $k\geq n+2$, and
\[ \theta(n+1) \defeq \twomat{1}{\om(n)}{0}{0} \quad,\quad
   \theta(n)=2\theta(n+1) = \twomat{2}{2\om(n)}{0}{0}  . \]
Then $\theta$ is an $\om$-bounded map $\Nmin\to \Mat{2}$, with
\[ \begin{aligned}
\norm{\theta}_{\infty,\om^{-1}}
 & = \max \left(
 \frac{\norm{\theta(n)} }{\om(n)} \ , \ \frac{\norm{\theta(n+1)} }{\om(n+1)} \right) \\
 & \leq
 \max \left( \frac{2+2\om(n)}{\om(n)}  \ , \ \frac{1+\om(n) }{\om(n+1)} \right)
 \leq \frac{2\om(n)}{\om(n+1)} .
\end{aligned} \]  

Clearly $\theta(j)\theta(k)=\theta(j)=\theta(k)\theta(j)$ whenever $1\leq j \leq n-1$ or $n+2\leq k$. Furthermore, $\theta(n+1)^2=\theta(n+1)$ and $\theta(n)\theta(n+1)=\theta(n)=\theta(n)\theta(n+1)$, while $\theta(n)^2-\theta(n) = 3\theta(n+1)$. Therefore
\[ \diff_\om(\theta) = \frac{ \norm{\theta(n)^2-\theta(n)} }{\om(n)^2} = \frac{3 \norm{\theta(n+1)} }{\om(n)^2} \leq \frac{6}{\om(n)} \leq \delta. \]

Now let $\phi:\Nmin\to \Mat{2}$ be a multiplicative function. Then $\phi(n)$ and $\phi(n+1)$ are commuting idempotents, so by taking $C=\theta(n+1)$ in Lemma~\ref{l:two-element-obstruction}(ii), 
\[ \text{either } \norm{\theta(n)-\phi(n)}\geq \om(n)/2 \text{ or } \norm{\theta(n+1)-\phi(n+1)}\geq \om(n)/4 \geq \om(n+1)/2.\]
Therefore
$\norm{\theta-\phi}_{\infty,\om^{-1}} \geq 1/2$.
\end{proof}


Theorems~\ref{t:wt-Nmin-T2} and \ref{t:wt-Nmin-M2} suggest that if we seek positive results, we are better off considering the unweighted convolution algebras of semilattices. This is corroborated by the final result of this section.

\begin{thm}\label{t:AMNM-T2}
Let $S$ be a semilattice. Then $(\ell^1(S),\sT_2)$ is a uniformly AMNM pair.
\end{thm}

\begin{proof}
Let $\theta:S\to \sT_2$ be a function satisfying $\diff(\theta) < 1/5$, and let $a,b:S\to\Cplx$ be the functions defined by
\[ \theta(e) = \twomat{a(e)}{b(e)}{0}{a(e)} \qquad(e\in S). \]
Since
\begin{equation}\label{eq:product-in-T2} 
\twomat{a(e)}{b(e)}{0}{a(e)}\twomat{a(f)}{b(f)}{0}{a(f)}=\twomat{a(e)a(f)}{a(e)b(f)+b(e)a(f)}{0}{a(e)a(f)}
\end{equation}
we see that $\diff(a) \leq \diff(\theta) < 1/5$.
By Proposition~\ref{p:scalars}, there exists a multiplicative function $\chi: S \to \Cplx$ such that
\[ \abs{a(e)-\chi(e)} < \frac{7}{5}\abs{a(e)^2-a(e)} < \frac{7}{5}\diff(\theta) < 7/25 \quad\text{for all $e\in S$.} \]
Note that
\[ \begin{aligned}
\diff(\theta) & \geq \sup_{e \in S} \norm{\theta(e)^2-\theta(e)} \\
& = \sup_{e\in S} \Norm{ \twomat{a(e)^2-a(e)}{( 2a(e)-1 )b(e)}{0}{a(e)^2-a(e)} } \\
& = \sup_{e\in S} \abs{a(e)^2-a(e)} + \abs{2a(e)-1}\ \abs{b(e)}.
\end{aligned} \]
Let $e\in S$.
Since
$\abs{a(e)^2-a(e)}\leq \diff(a) \leq\diff(\theta)<1/5$, applying Lemma~\ref{l:in-C} and using the estimate \eqref{eq:rho(1/5)} yields
\[ \min(\abs{a(e)}, \abs{1-a(e)} ) < \rho\left(\frac{1}{5}\right) \diff(a) < \frac{7}{25};\]
therefore $a(e) \in \oD_0(7/25) \cup \oD_1(7/25)$, so that
\[  \abs{\frac{1}{2}-a(e)} \geq \frac{1}{2}-\frac{7}{25}  = \frac{11}{50} \,. \]
Hence
\[ \begin{aligned}
\norm{\theta(e)-\chi(e)I}
& = \abs{a(e)-\chi(e)}+\abs{b(e)} \\
& \leq \frac{7}{5}\abs{a(e)^2-a(e)} + \frac{25}{11}\abs{(2a(e)-1)b(e)} \leq \frac{25}{11}\diff(\theta) ,
\end{aligned} \]
as required.
\end{proof}

It is now natural to ask if $(\ell^1(S),\Mat{2})$ is an AMNM pair. The answer turns out to be yes -- in fact, it is always a \emph{uniformly} AMNM pair -- but the proof is considerably harder, and occupies all of the next section.
\end{section}

\begin{section}{$(\ell^1(S),\Mat{2})$ is AMNM for any semilattice $S$}\label{s:AMNM-M2}

For reasons of technical convenience, we shall work mostly with the 
Hilbert-Schmidt norm on $\Mat{2}$, defined by $\hsnorm{A}^2 = \tr(A^*A)^{1/2}$. It might be conceptually clearer to use the operator norm throughout, but this seems to yield worse constants in later inequalities, which are obtained by bootstrapping up the earlier ones.

\begin{notn}
If $A\in \Mat{2}$ and $\veps>0$ let
\[ \clB_A(\veps) = \{ B \in \Mat{2} \st \hsnorm{A-B} \leq \veps \} .\]
(the closed ball of radius $\veps$ centred on $A$, in the Hilbert-Schmidt norm).
\end{notn}

\begin{thm}\label{t:slatt-M2-amnm}
Let $\delta< 0.03$, and let $\theta: S \to \Mat{2}$ be a function satisfying
\begin{equation}\label{eq:HS-delta-mult}
\sup_{e.f\in S} \hsnorm{\theta(e)\theta(f)-\theta(ef)} \leq \delta .
\end{equation}
Then there exists a (bounded) multiplicative function $\phi: S\to \Mat{2}$ such that
\[ \sup_{x\in S} \hsnorm{\theta(x)-\phi(x)} \leq 12 \delta.\]
In particular, $(\ell^1(S),\Mat{2})$ is a uniformly AMNM pair.
\end{thm}

Note that we do not assume in \eqref{eq:HS-delta-mult} that $\sup_{e\in S}\hsnorm{\theta(e)}< \infty$, but that this will emerge during the proof (see Proposition~\ref{p:separate-S1}).

\subsection{Motivating the proof of Theorem~\ref{t:slatt-M2-amnm}}
Let $S$ be a semilattice. Our proof that $(\ell^1(S),\Cplx)$ is a uniformly AMNM pair can be broken down into three steps:
\begin{enumerate}
\item Show that there is a constant $c$, such that whenever $\theta:S\to\Cplx$ is $\delta$-multiplicative, $\theta(S)\subseteq \clD_0(c\delta)\cup \clD_1(c\delta)$.
(In effect, this step ``approximately discretizes'' the problem.)
\item Put $S_k\defeq \theta^{-1}(\clD_k(c\delta))$ for $k=0,1$, so that $S$ is partitioned as $S_1\cup S_0$. Show that $S_1\cdot S_1 \subseteq S_1$, $S_1\cdot S_0 \subseteq S_0$, and $S_0\cdot S_0 \subseteq S_0$.
(Although we did not do these calculations explicitly, they are implicit in the process of checking $S_1$ is a filter in~$S$.)
\item Define $\phi: S \to \{0,1\}$ by $\phi=1$ on $S_1$ and $\phi=0$ on $S_0$. By the first step, $\norm{\phi-\theta}_\infty\leq c\delta$, and by the second step, $\phi$ is multiplicative.
\end{enumerate}

The strategy we shall adopt is to mimic each of these steps, but now allow our maps to take values in $\Mat{2}$ rather than $\Cplx$.
As a first step, we need some characterization of multiplicative functions $\phi:S\to \Mat{2}$, which reduces down to the problem of describing the possible semilattices inside the multiplicative semigroup $\Mat{2}$. This is not too hard, once we make the following observation:
\textit{if $P\in\Mat{2}$ is a rank-one idempotent, then the only idempotents which commute with $P$ are $I$, $P$, $I-P$ and $0$.}
(We will see later that there is an ``approximate version'' of this.)

Secondly, observe that if $\theta$ is -- as claimed -- a perturbation of a multiplicative function~$\phi$, then by the previous remarks $\theta(S)$ should be contained in a small-ball neighbourhood of $\phi(S)$, which in turn is contained in a set of at most four commuting idempotents.
To prove Theorem~\ref{t:slatt-M2-amnm}, we reverse this line of reasoning, and identify a commuting set $L$ of idempotents in $\Mat{2}$, a small-ball neighbourhood of which will contain $\theta(S)$.
(See Proposition~\ref{p:separate-S1} for the details.) Then, since the elements of $L$ are well-separated, there is only one realistic candidate for the map $\phi$: namely, it should send a given $x\in S$ to the element of $L$ nearest to $\theta(x)$.
This map $\phi$ will clearly be close in norm to $\theta$, so all that will remain is to check that $\phi$ is multiplicative: this can be done through a case-by-case analysis, although some work is needed since we do not assume $\sup_{e\in S}\hsnorm{\theta(e)}<\infty$.

To identify the set $L$, we make heavy use of a small but technical result, based on the following idea: the trace of an approximately idempotent $2\times 2$ matrix must be close to an integer, which then equals the rank of any nearby idempotent. This will be made precise in the next lemma.

\subsection{A key technical lemma} 
The following lemma is our basic tool for working with approximately idempotent elements of $\Mat{2}$. It is here that the Hilbert-Schmidt norm seems to be convenient.

\begin{lem}[Key estimates]\label{l:2by2-key}
\
\begin{itemize}
\item[\rm(a)]
Let $\rho(0)= 1$ and
\[ \rho(t)= \frac{1}{2t}\left(1-\sqrt{1-4t}\right) \qquad(0 <t \leq 1/4)\]
and let
\[
\kp(t) = \left( 1- \rho(t)t\sqrt{2}\right)^{-1}  \qquad(0\leq t \leq 1/4).
\]
Then $\rho$ and $\kp$ are increasing functions, with $\rho(t)\leq\kp(t)$ for all $t\in [0,1/4]$. Moreover,
\begin{equation}\label{eq:nice-rho}
\rho\left(\frac{n}{(n+1)^2}\right) = \frac{n+1}{n}
\end{equation}
and
\begin{equation}\label{eq:bound_regulator}
 \kp\left(\frac{n}{(n+1)^2}\right) = 
 \left( 1- \frac{\sqrt{2}}{n+1}\right)^{-1} < \left(1- \frac{10}{7(n+1)}\right)^{-1} .
\end{equation}
\item[\rm(b)]
Let $A\in \Mat{2}$ satisfy $\hsnorm{A-A^2}\leq\veps  < 2/9$. Then
\begin{equation}\label{eq:2A-I}
\hsnorm{2A-I} \geq (2-6\hsnorm{A-A^2})^{1/2}\,,
\end{equation}
and
\begin{equation}\label{eq:coarse-confine-trace}
 \tr(A) \in \bigcup_{j\in \{0,1,2\}} \clD_j(\sqrt{2}\rho(\veps)\veps) \subseteq \bigcup_{j\in \{0,1,2\}} \clD_j(10/21). 
\end{equation}
Moreover:
\begin{itemize}
\item
 if $\abs{\tr(A)-2} < 1/2$, then $\hsnorm{I-A}\leq \kp(\veps)\veps$;
\item
 if $\abs{\tr(A)} < 1/2$, then $\hsnorm{A}\leq \kp(\veps)\veps$;
\item
 if $\abs{\tr(A)-1}<1/2$ then there exists a rank-one idempotent $P\in\Mat{2}$ satisfying $\hsnorm{I-P}\leq \rho(\veps)\veps$.
\end{itemize}
\end{itemize}
\end{lem}

\begin{proof}[Proof of (a)]
The formulas \eqref{eq:nice-rho} and \eqref{eq:bound_regulator} follow from the definitions of the functions $\rho$ and $\kp$ by direct calculation.
Moreover: we already saw (see the remarks after the formula \eqref{eq:key-function}) that $\rho$ is increasing on $[0,1/4]$, and therefore $\kp$ is also increasing on $[0,1/4]$. 

It only remains to prove that $\rho(t)\leq\kp(t)$ for all $t\in [0,1/4$]. Since $1=\kp(0)=\rho(0) < \rho(t)$ for all $0 < t\leq 1/4$, it suffices to show that $\kp(t)^{-1} < \rho(t)^{-1}$ for all such $t$. To do this, observe that
\[ \begin{aligned}
\kp(t)^{-1}+\sqrt{2}-1
 & = \sqrt{2}-\sqrt{2}\rho(t)t \\
 & = \sqrt{2}\left( 1- \frac{1}{2}\left(1-\sqrt{1-4t}\right)\right) \\
 & = \sqrt{2}\left(\frac{1}{2}+\frac{1}{2}\sqrt{1-4t}\right),
\end{aligned} \]
while
\[ \left(1+\sqrt{1-4t}\right)\rho(t) = \frac{1}{2t} \left(1 - (1-4t) \right) = 2 ; \]
combining these two identities yields
\[ \kp(t)^{-1}+\sqrt{2}-1 = \sqrt{2}\rho(t)^{-1}, \]
so that
\[ \kp(t)^{-1}-\rho(t)^{-1} = (\sqrt{2}-1) \rho(t)^{-1} - (\sqrt{2}-1) < 0 \]
as required.
\end{proof}

\begin{proof}[Proof of (b)]
By conjugating with an appropriate unitary matrix, we may assume \WLOG\ that $A$ is upper triangular, say
$A = \twomat{a}{b}{0}{d}$. Then
\[ A-A^2 = \twomat{a-a^2}{b(1-a-d)}{0}{d-d^2} \]
so that
\begin{equation}\label{eq:ORLY}
\abs{a-a^2}^2 + \abs{b(1-a-d)}^2 + \abs{d-d^2}^2 \leq
\hsnorm{A-A^2}^2  \leq \veps^2.
\end{equation}
Therefore,
\[ \begin{aligned}
\hsnorm{2A-I}^2 
 & \geq \abs{2a-1}^2+\abs{2d-1}^2 \\
 & \geq 2 - 4\abs{a-a^2} - 4\abs{d-d^2} \\
 & \geq 2 - 4\sqrt{2}\left(\abs{a-a^2}^2+\abs{d-d^2}^2\right)^{1/2} \\
 & \geq 2 - 6\left(\abs{a-a^2}^2+\abs{d-d^2}^2\right)^{1/2} \\
 & \geq 2- 6 \hsnorm{A-A^2},
\end{aligned} \]
and we have proved \eqref{eq:2A-I}.

It also follows from \eqref{eq:ORLY}, by using Lemma~\ref{l:in-C}, that
\[ \begin{aligned}
& \dist_{\Cplx}(a,\{0,1\}) \leq \rho(\veps)\abs{a-a^2} & \leq \rho\left(\frac{2}{9}\right) \frac{2}{9} = \frac{1}{3} , \\
\text{and} & \dist_{\Cplx}(d,\{0,1\}) \leq \rho(\veps)\abs{d-d^2} & \leq \rho\left(\frac{2}{9}\right) \frac{2}{9} = \frac{1}{3}.
\end{aligned} \]

Define the function $N: \clD_0(1/3) \cup \clD_1(1/3) \to \{0,1\}$ to take the value $i$ on $\clD_i(1/3)$ for $i=0,1$. By Cauchy--Schwarz,
\[ \begin{aligned}
\abs{a-N(a)}+\abs{d-N(d)}
 & \leq \rho(\veps) ( \abs{a-a^2} +\abs{d-d^2} ) \\
 & \leq \rho(\veps)\sqrt{2}  ( \abs{a-a^2}^2 +\abs{d-d^2}^2 )^{1/2} \leq \rho(\veps)\sqrt{2}\veps\,;
\end{aligned} \]
and since $\rho$ is an increasing function,
\[ \rho(\veps)\sqrt{2}\veps \leq \frac{10}{7}\rho(\veps)\veps \leq \frac{10}{7}\rho\left(\frac{2}{9}\right)\frac{2}{9} = \frac{10}{21}. \]

Observe that if $r\in\Zahl$ satisfies $\abs{\tr(A)-r}\leq 1/2$, then
\[ \abs{r-N(a)-N(d)} < \frac{1}{2}+\frac{10}{21} < 1,\]
forcing $r=N(a)+N(d)$.
So $r\in\{0,1,2\}$, and we always have
\begin{equation}\label{eq:il-faut-retourner}
 \abs{\tr(A)- r} = \abs{a+d-N(a)-N(d)} < \rho(\veps)\sqrt{2}\veps \leq \frac{10}{21}\,,
\tag{$*$}
\end{equation}
giving us the inclusions in \eqref{eq:coarse-confine-trace}.

\medskip
Finally, we show that $A$ is always $\hsnorm{\cdot}$-close to an idempotent of the appropriate rank.
We first address the cases where exactly one of $N(a)$ and $N(d)$ is equal to $1$, with the other being equal to $0$. In both of these cases we define
\[ P = \twomat{N(a)}{b}{0}{N(d)} . \]
A small calculation shows that $P=P^2$. By construction,
\[  \begin{aligned}
\hsnorm{A-P}
 & = \left(\abs{a-N(a)}^2+\abs{d-N(d)}^2\right)^{1/2} \\
 & \leq \rho(\veps) \left( \abs{a-a^2}^2 + \abs{d-d^2}^2\right)^{1/2} \leq \rho(\veps)\hsnorm{A-A^2}.
\end{aligned} \]

Secondly, we address the cases where $\tr(A)\in \oD_0(1/2)\cup \oD_2(1/2)$.
If $\abs{\tr(A)-2}<1/2$ we must have $N(a)=N(d)=1$ (so $D=I$) and then, using \eqref{eq:il-faut-retourner},
\[
\abs{1-a-d} \geq 1- \abs{a+d-2} =  1- \abs{a+d-N(a)-N(d)}. 
\]
On the other hand, if $\abs{\tr(A)}<1/2$, we must have $N(a)=N(d)=0$ (so $D=0$) and then
\[
\abs{1-a-d} \geq 1- \abs{a+d} = 1- \abs{a+d-N(a)-N(d)}.
\]
Thus, in both of these cases, by using \eqref{eq:il-faut-retourner} we obtain the inequality
\begin{equation}\label{eq:BREL}
\left(\abs{1-a-d}\right)^{-1} \leq \left(1- \abs{a+d-N(a)-N(d)}\right)^{-1} \leq \left(1-\rho(\veps)\veps\sqrt{2}\right)^{-1} =\kp(\veps).
\tag{$**$}
\end{equation}
Put
\[ D= \twomat{N(a)}{0}{0}{N(d)}; \]
this matrix equals either $I$ or $0$, depending on whether $\tr(A)$ is close to $0$ or close to $2$. It follows from \eqref{eq:BREL} that
\[ \begin{aligned}
\hsnorm{A-D}^2
 & = \abs{a-N(a)}^2  + \abs{d-N(d)}^2 + \abs{b}^2\\
 & \leq \rho(\veps)^2 \abs{a-a^2}^2 + \rho(\veps)^2 \abs{d-d^2}^2 + \kp(\veps)^2\abs{b(1-a-d)}^2 \\
 & \leq \kp(\veps)^2( \abs{a-a^2}^2 + \abs{d-d^2}^2 + \abs{b(1-a-d)}^2),
\end{aligned}  \]
where we used the result from part (a) that $\rho(\veps)\leq \kp(\veps)$.
Taking square roots gives $\hsnorm{A-D} \leq \kp(\veps) \hsnorm{A-A^2} \leq \kp(\veps)\veps$,
and the proof of our lemma is complete.
\end{proof}

\begin{figure}[hptb]
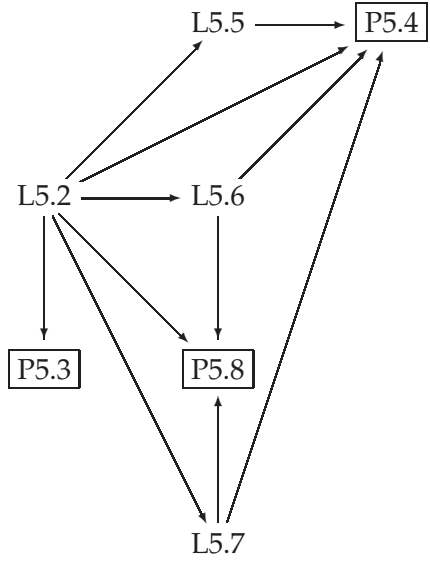

\begin{center}

\end{center}
\begin{diagram}
& & \hbox{L\ref{l:separated}} &  \rTo &  \fbox{P\ref{p:separate-S1}}   \\
 & \ruTo &  &  \ruTo(4,2) \ruTo \ruTo(2,6)  \\
\hbox{L\ref{l:2by2-key}}  & \rTo & \hbox{L\ref{l:S2-an-upset}} &   &   \\
 \dTo & \rdTo(2,4)\rdTo & \dTo &   &\\
\fbox{P\ref{p:basics}} & & \fbox{P\ref{p:S2-times-S1}} & &  \\
 & & \uTo &  & \\
 & & \hbox{L\ref{l:chains-in-S1}} & &
\end{diagram}
\caption{Dependencies between results in this section. (P = Proposition; L = Lemma.)}
\end{figure}

\subsection{The proof of Theorem~\ref{t:slatt-M2-amnm}}
To save needless repetition, we will assume for the rest of this section that $0\leq\delta < 0.03$.
This implies from the outset (by Lemma~\ref{l:2by2-key}(a)) that
\[ \rho(\delta) \leq \kp(\delta)  < \kp\left(\frac{29}{30^2}\right) < \left(1-\frac{1}{21}\right)^{-1} = 1.05. \]

Let $\theta:S\to\Mat{2}$ satisfy the condition~\eqref{eq:HS-delta-mult}.
To simplify formulas, we shall use the following abbreviations: $\te$ stands for $\theta(e)$, $\tf$ for $\theta(f)$, $\tef$ for $\theta(ef)$, and so forth.

We start by taking $e=f$ in condition~\eqref{eq:HS-delta-mult} and using Lemma~\ref{l:2by2-key} with $\veps=\delta$. This gives
\begin{equation}\label{eq:trace-localized}
\begin{aligned}
\tr\theta(S)
 \subseteq \bigcup_{j\in\{0,1,2\}} \clD_j(\rho(\delta)\sqrt{2}\delta) 
 \subseteq \bigcup_{j\in\{0,1,2\}} \oD_j(0.05) 
\end{aligned}
\end{equation}
with the second inclusion following from the upper bound
\[ \rho(\delta)\sqrt{2}\delta 
<  1.05 \times \frac{10}{7} \times 0.03 < 0.05 .\]

For $k=0,1,2$, define
\begin{equation}\label{eq:attracted}
 S_k\defeq
 (\tr\circ\theta)^{-1}(\oD_k(0.95)) 
= (\tr\circ\theta)^{-1}\left(\clD_k(\rho(\delta)\sqrt{2}\delta)\right).
\end{equation}
Then, by Lemma~\ref{l:2by2-key}(b),
$\theta(S_2) \subseteq \clB_I(\kp(\delta)\delta)$
and
$\theta(S_0) \subseteq  \clB_0(\kp(\delta)\delta)$.

Define $\phi:S_2\sqcup S_0 \to \Mat{2}$ by
$\phi(S_0) = \{0\}$ and $\phi(S_2) = \{I\}$.  
By our initial remarks concerning $\kp(\delta)$,
\begin{equation}
\label{eq:easy-extremes}
\sup_{x\in S_2\sqcup S_0} \hsnorm{\tx-\phi(x)}\leq \kp(\delta)\delta \leq 1.05\delta.
\end{equation}

\begin{prop}[Some easy properties]\label{p:basics}
Let $e,f\in S$.
\begin{YCnum}
\item\label{li:inverse-bound}
If $f\in S_2$ then $\hsnorm{\tf^{-1}} < 1.5$. If $e\in S_0$ then $\hsnorm{(I-\te)^{-1}} < 1.5$.
\item\label{li:S2-times-S2}
 If $e,f\in S_2$ then $ef\in S_2$.
\item\label{li:S0-is-ideal}
If $e\in S_0$ and $f\in S$ then $ef\in S_0$.
\end{YCnum}
In particular, $S_2\sqcup S_0$ is a subsemigroup of $S$, and $\phi:S_2\sqcup S_0\to\Mat{2}$ is multiplicative.
\end{prop}

\begin{proof}
Suppose $f\in S_2$. As remarked above, we have $\hsnorm{\tf-I} \leq \kp(\delta)\delta \leq 1.05\delta < 0.04$. Therefore, using submultiplicativity of the Hilbert--Schmidt norm,
\begin{equation}\label{eq:inverse-bound}
\begin{aligned}
\hsnorm{\tf^{-1}}
 & = \hsnorm{ \sum_{n=0}^\infty (I-\tf)^n } \\
 & \leq \hsnorm{I} + \sum_{n=1}^\infty  \hsnorm{I-\tf}^n  \\
 & \leq \sqrt{2} -1 +  \frac{1}{1-\hsnorm{I-\tf}}
 & < \frac{3}{7} + \frac{1}{0.96}   < 1.5.
\end{aligned}
\end{equation}
An exactly similar argument shows that when $e\in S_0$, we have $\hsnorm{(I-\te)^{-1}} < 1.5$.
This proves part~\ref{li:inverse-bound}.

Now, for any $e,f\in S$,
\[ \hsnorm{(\tef-\te)\tf} \leq \hsnorm{\tef\tf-\tef}+\hsnorm{\tef-\te\tf} \leq 2\delta.\]
If $e,f\in S_2$, we may combine this upper bound with part~\ref{li:inverse-bound} to obtain (via Cauchy--Schwarz)
\[ \begin{aligned}
 \abs{\tr(\tef)-2}
 & \leq \abs{\tr(\te)-2} + \abs{\tr(\tef-\te)} \\
 & \leq 0.05 + \hsnorm{(\tef-\te)\tf}\hsnorm{\tf^{-1}} \leq 0.05+ 3\delta \ll 0.95,
\end{aligned}
\]
which implies $ef\in S_2$, by the definition in~\eqref{eq:attracted}. This proves~\ref{li:S2-times-S2}.

Observe that $\hsnorm{\te\ \tef-\tef} \leq \delta$. If $e\in S_0$, then by part~\ref{li:inverse-bound}
$\hsnorm{(\te -I)^{-1}} < 1.5$,
and so by Cauchy-Schwarz,
\[ \abs{ \tr \tef } \leq \hsnorm{ (\te-I)^{-1} }\ 
 \hsnorm{(\te-I)\tef} \leq 1.5\delta \ll 0.5. \]
Hence $ef\in S_0$, by \eqref{eq:attracted}, proving \ref{li:S0-is-ideal}. The final statement of the proposition now follows easily from~\ref{li:S2-times-S2} and~\ref{li:S0-is-ideal}.
\end{proof}

Combining \eqref{eq:easy-extremes} and Proposition~\ref{p:basics}, we get a proof of Theorem~\ref{t:slatt-M2-amnm} in the special case where $S_1=\emptyset$. \emph{We shall therefore assume, for the rest of this section, that $S_1$ is non-empty.}

\begin{prop}\label{p:separate-S1}
There exists a rank $1$-idempotent $P\in\Mat{2}$ such that
\[ \theta(S_1) \subseteq \clB_P(12\delta) \sqcup \clB_{I-P}(12\delta). \]
Moreover, given $e,f\in S_1$:
\begin{itemize}
\item[--] if $\te,\tf$ both lie in $\clB_P(12\delta)$, then so does $\tef$;
\item[--] if $\te,\tf$ both lie in $\clB_{I-P}(12\delta)$, then so does $\tef$;
\item[--] if $\te\in\clB_P(12\delta)$ and $\tf\in\clB_{I-P}(12\delta)$, then $ef\in S_0$.
\end{itemize}
\end{prop}

The proof of Proposition~\ref{p:separate-S1} requires some work, which we break up into several lemmas.
The first of these is an ``approximate version'' of the following observation:

\medskip
if $P$ and $Q$ are rank-one idempotents in $\Mat{2}$ with $PQ=0=QP$, then $P+Q=I$.
\medskip

\noindent
This is also the first place where it is really necessary to make $\delta$ no bigger than about $0.03$, as we need to apply Lemma~\ref{l:2by2-key} to something which is ``approximately idempotent to within roughly $7\delta$''.

\begin{lem}[Separation lemma]\label{l:separated}
Let $e,f\in S_1$ with $ef\in S_0$.
 Then
\[ \hsnorm{\te+\tf-I} \leq 10\delta \quad\text{and}\quad 
\hsnorm{\te-\tf} \geq \frac{4}{3} - 10\delta > 1.\]
\end{lem}

\begin{proof}
Let $B= \te+\tf$. We wish to prove $B$ is close to $I$; this will follow if we can show $B$ is approximately idempotent and has trace close to~$2$.

Since $ef\in S_0$, applying Lemma~\ref{l:2by2-key}(b) to the matrix $\tef$, with $\veps=\delta$, yields
$\hsnorm{\tef} \leq \kp(\delta)\delta$.
Then, since 
\[ \hsnorm{\te\tf+\tf\te} \leq \hsnorm{\te\tf-\tef}+\hsnorm{\tf\te-\tef} + 2\hsnorm{\tef} \leq 2\delta+2\kp(\delta)\delta,\]
we obtain
\[ \hsnorm{B^2-B} \leq \hsnorm{\te^2-\te} + \hsnorm{\tf^2-\tf} + \hsnorm{\te\tf+\tf\te} \leq (4+2\kp(\delta))\delta. \]

As $(4+2\kp(\delta))\delta \leq 6.1\delta < 0.183 < 2/9$,
we may apply Lemma~\ref{l:2by2-key}(b) to the  matrix $B$.
By our earlier observation \eqref{eq:trace-localized}, 
\[ \abs{\tr(B)-2} \leq \abs{\tr(\te)-1}+\abs{\tr(\tf-1)} < 0.05+0.05 < \frac{1}{2}, \]
and so (by Lemma~\ref{l:2by2-key}(b)), we have
\begin{equation}\label{eq:muppet}
\hsnorm{B-I}\leq \kp(6.1\delta)6.1\delta
\leq \kp\left(\frac{3}{4^2}\right)6.1\delta
\leq \left( 1- \frac{5}{14}\right)^{-1} 6.1\delta  \leq 10\delta .
\tag{$*$}
\end{equation}

For the second part, we apply the estimate \eqref{eq:2A-I} to obtain
\[ \hsnorm{2\te-I }\geq (2-6\delta)^{1/2} \geq \sqrt{1.82} > \frac{4}{3}.\]
Combining this with \eqref{eq:muppet} yields
\[
\hsnorm{\te-\tf} = \hsnorm{(2\te-I) - (B-I)}
  \geq \frac{4}{3} - 10\delta  > 1,
\]
as required.
\end{proof}

Intuitively, we should have $S_1\cdot S_1\subseteq (S_1\cup S_0)$ since elements of $\theta(S_k)$ are close to idempotents of rank $k$. Some care is needed to show this, because we have no {\it a~priori}\/ upper bound on norms of elements in $\theta(S_1)$.

\begin{lem}\label{l:S2-an-upset}\
\begin{YCnum}
\item If $e\in S$, $f\in S_2$ and $e\succeq f$, then $e\in S_2$.
\item $S_1\cdot S_1 \subseteq S_1\cup S_0$.
\end{YCnum}
\end{lem}

\begin{proof}
Let $e,f\in S$ with $ef=f$. Then
$\hsnorm{\te\tf-\tf}\leq \delta$. If $f\in S_2$, then just as in the proof of Proposition~\ref{p:basics}, we have $\hsnorm{\tf^{-1}} < 1.5$. 
Hence, by Cauchy--Schwarz,
\[ \abs{ \tr(\te-I) } \leq \hsnorm{\te\tf-\tf}
 \hsnorm{\tf^{-1}} \leq 1.5\delta \ll \frac{1}{2}, \]
forcing $e$ to lie in $S_2$. This proves~(i).

Now if $e,g\in S_1$, put $f=eg$; by part~(i), $f\notin S_2$, and (ii) is proved.
\end{proof}

To analyze $S_1$ in further detail, the following lemma is useful.

\begin{lem}[Chains in $S_1$]
\label{l:chains-in-S1}
Let $e,f\in S_1$ with $e\succeq f$.
Then $\hsnorm{\te-\tf}\leq 5\delta$.
\end{lem}

\begin{proof}
Put $A=\te-\tf$. We show that $A$ is approximately idempotent and has small trace, so must be close to $0$ by Lemma~\ref{l:2by2-key}(b).
In detail: observe that
\[ \hsnorm{A^2-A} =\hsnorm{
{\te}^2-\te + {\tf}^2-\tf 
- \te\tf +\tf -\tf\te+\tf 
} \leq 4\delta < 0.12, \]
and $\tr(A) \leq \abs{1-\tr\te}+\abs{1-\tr\tf} \leq 0.05+0.05 \ll 0.5$.
Applying Lemma~\ref{l:2by2-key}(b) to the matrix $A$, we have
$\hsnorm{\te-\tf} \leq \kp(4\delta)4\delta \leq \kp(0.12)4\delta$.

Calculation shows that
\[ \begin{aligned}
1- \kp(0.12)^{-1}
 & = \sqrt{2}\rho(0.12)0.12  \\
 & = \frac{1}{\sqrt{2}}(1- \sqrt{1-0.48})  \\
 & = \frac{1}{\sqrt{2}} - \sqrt{0.26}   & < 0.2,
\end{aligned} \]
so that $4\kp(0.12) < 4(1-0.2)^{-1} = 5$.
The rest follows.
\end{proof}

\begin{proof}[Proof of Proposition~\ref{p:separate-S1}]
Recall that $S_1$ is, by assumption, non-empty.
Consider the relation on $S_1$ defined by $\{ (e,f) \in S_1\times S_1 \st ef\in S_1\}$, and denote it by $\sim$. Clearly $\sim$ is symmetric and reflexive.
We will see shortly that it is also transitive, as a consequence of the following two observations:
\begin{YCnum}
\item\label{li:equiv-implies-close-images}
If $e\sim f$ then $\hsnorm{\te-\tf}\leq 10\delta$.
(For if $e,f,ef\in S_1$ then Lemma~\ref{l:chains-in-S1} implies that $\hsnorm{\te-\tef}\leq 5\delta$ and $\hsnorm{\tf-\tef} \leq 5 \delta$.)
\item\label{li:vaguely-close-implies-equiv-preimages}
If $e,f\in S_1$ and $\hsnorm{\te-\tf}\leq 1$, then $e\sim f$.
(This is immediate from the contrapositive of the separation lemma (Lemma~\ref{l:separated}).)
\end{YCnum}
Therefore, since $20\delta < 1$, we see that $\sim$ is indeed transitive, and so $\sim$ is an equivalence relation on $S_1$.

Now let $e,f,g\in S_1$. Suppose $e\not\sim f$ and $f\not\sim g$.
 Then, since $ef\in S_0$ and $fg\in S_0$, the separation lemma (Lemma~\ref{l:separated}) implies that
 $\hsnorm{\te+\tf-I}\leq 10\delta$
and $\hsnorm{\tf+\tg-I}\leq 10\delta$.
 Hence
\[ \hsnorm{\te-\tg}
 \leq \hsnorm{\te+\tf-I} + \hsnorm{\tf+\tg-I}
 \leq 20 \delta \leq 0.6 ; \]
so by \ref{li:vaguely-close-implies-equiv-preimages}, $e\sim g$. Thus there are at most two equivalence classes for this relation.

Now, fix $p_0\in S_1$. By \ref{li:equiv-implies-close-images},
\[ \theta([p_0]) \subseteq \clB_{\thp_0}(10\delta), \]
where $[p_0]$ denotes the equivalence class of $p_0$ in~$S_1$.
Moreover, if $e\in S_1\setminus [p_0]$ then $ep_0\in S_0$ (by definition of $\sim$).
Hence, by the separation lemma (Lemma~\ref{l:separated}), $\hsnorm{\te+\thp_0-I} \leq 10\delta$, so that
\[ \theta(S_1\setminus [p_0]) \subseteq \clB_{I-\thp_0}(10\delta). \]

By Lemma~\ref{l:2by2-key}(b), there exists a rank-$1$ idempotent $P\in\Mat{2}$ such that $\hsnorm{P-\thp_0} \leq 1.05\delta$. Then
\[ \theta([p_0]) \subseteq \clB_P(10\delta+1.05\delta) \subseteq \clB_P(12\delta) \]
and
\[ \theta(S_1\setminus [p_0]) \subseteq \clB_{I-P}(10\delta+1.05\delta)
\subseteq \clB_{I-P}(12\delta) \]
Note that since $\hsnorm{2P-I} \geq \sqrt{2}$ (by Lemma~\ref{l:2by2-key}(b))
 and $24\delta \ll \sqrt{2}$, the sets $\theta([p_0])$ and $\theta(S_1\setminus [p_0])$ are disjoint.
Therefore,
\[
S_1 \cap \theta^{-1}\left( \clB_P(12\delta) \right) = [p_0]
\quad\text{and}\quad
S_1 \cap \theta^{-1}\left( \clB_{I-P}(12\delta) \right) = S_1\setminus [p_0]
\]

Finally, let $e,f\in S_1$.
\begin{itemize}
\item[--] If $\te,\tf\in \clB_P(12\delta)$ then $e\sim p_0\sim f$, so that $ef\sim p_0$ (as equivalence classes are closed under multiplication) and therefore $\tef\in \clB_{P}(12\delta)$.
\item[--] If $\te,\tf\in \clB_{I-P}(12\delta)$, then $e\not\sim p_0$, $f\not\sim p_0$; since there are at most two equivalence classes, $e\sim f$. Thus $e\sim ef\sim f$, so $ef\not\sim p_0$, so $\tef\in \clB_{I-P}(12\delta)$.
\item[--] If $\te\in \clB_P(12\delta)$ and $\tf\in \clB_{I-P}(12\delta)$, then $e\sim p_0$ and $p_0\not\sim f$. Thus $e\not\sim f$, so $ef\in S_0$.
\end{itemize}
This completes the proof of the proposition.
\end{proof}

\medskip
We now \emph{fix} an idempotent $P\in\Mat{2}$ that satisfies the conclusions of Proposition~\ref{p:separate-S1}.
The sets
\[ \clB_I(2\delta) \quad,\quad \clB_P(12\delta) \quad,\quad \clB_{I-P}(12\delta) \quad, \quad \clB_0(2\delta) \]
are pairwise disjoint, and their union contains $\theta(S)$. Recall that we have already defined $\phi:S_2\sqcup S_0 \to \Mat{2}$ which is multiplicative; now define $\phi:S_1\to \Mat{2}$ by setting $\phi(x)$ to be whichever of $P$ and $I-P$ is closer to $\tx$. Explicitly, the map $\phi:S\to\Mat{2}$ satisfies:
\[ \phi(x) = \left\{ \begin{aligned}
 I & \quad\text{if $x\in S_2$,} \\
 P & \quad\text{if $\tx\in \clB_P(12\delta)$,} \\
I-P & \quad\text{if $\tx\in \clB_{I-P}(12\delta)$,} \\
 0 & \quad\text{if $x\in S_0$.} \\
\end{aligned} \right. \]
By construction, $\hsnorm{\tx-\phi(x)}\leq 12\delta$ for all $x\in S$.

It remains to show $\phi$ is multiplicative.
Let 
\[ S_p\defeq \theta^{-1}(\clB_P(12\delta)) \quad\text{and}\quad
   S_q\defeq \theta^{-1}(\clB_{I-P}(12\delta)). \]
Then $S_1=S_p\sqcup S_q$, and it suffices to verify the following claims:

\begin{enumerate}
\item\label{li:2-2}
 $S_2\cdot S_2\subseteq S_2$.
\item\label{li:arb-0}
 $S\cdot S_0 \subseteq S_0$.
\item\label{li:2-p_2-q}
$S_2\cdot S_p\subseteq S_p$ and $S_2\cdot S_q\subseteq S_q$.
\item\label{li:p-p_q-q}
 $S_p\cdot S_p \subseteq S_p$ and $S_q\cdot S_q\subseteq S_q$.
\item\label{li:p-q}
 $S_p\cdot S_q \subseteq S_0$.
\end{enumerate}

Assertions~\ref{li:2-2} and~\ref{li:arb-0} follow from Proposition~\ref{p:basics}. Assertions \ref{li:p-p_q-q} and \ref{li:p-q} follow from Proposition~\ref{p:separate-S1}.
Assertion~\ref{li:2-p_2-q} requires some more work, and is dealt with in our final proposition.

\begin{prop}\label{p:S2-times-S1}
 Let $e\in S_2$. If $f\in S_p$, then so is $ef$; if $f\in S_q$, then so is $ef$.
\end{prop}

\begin{rem}
Although we know that $\te$ is close to $I$ for each $e\in S_2$, there might still exist rank-$1$ idempotents $R\in\Mat{2}$ such that $\norm{\te \cdot R - R}$ is large. So while the conclusion of Proposition~\ref{p:S2-times-S1} is as one would expect, the proof is somewhat circuitous.
\end{rem}

\begin{proof}[Proof of Proposition~\ref{p:S2-times-S1}]
Let $e\in S_2$ and $f\in S_1$. By Lemma~\ref{l:S2-an-upset}, $ef\notin S_2$. We claim that $ef\notin S_0$. For, assume $ef\in S_0$: then by Lemma~\ref{l:2by2-key}(b), $\hsnorm{\tef}\leq \kp(\delta)\delta \leq 1.05\delta$.
Since $e\in S_2$, we have $\hsnorm{(\te)^{-1}}<1.5$, by the same argument as in the proof of Proposition~\ref{p:basics}. Therefore, by Cauchy--Schwarz,
\[ \abs{\tr\tf} \leq \hsnorm{(\te)^{-1}}\hsnorm{\te\tf} \leq 1.5(\delta+ \hsnorm{\tef})  \ll 0.95 . \]
But, by \eqref{eq:attracted}, this implies $f\in S_0$, contradicting the assumption that $f\in S_1$.

The only remaining possibility is that $ef\in S_1$, and hence by Lemma~\ref{l:chains-in-S1} we have
\[ \hsnorm{\tf-\tef} \leq 5\delta \leq 0.15. \]
On the other hand, 
recalling that $\hsnorm{2P-I}\geq \sqrt{2}$
(by Lemma~\ref{l:2by2-key}(b)), we see that the distance between $\clB_P(12\delta)$ and $\clB_{I-P}(12\delta)$ is bounded below by $\sqrt{2}-24\delta > 0.6$.
Therefore, $\tf$ and $\tef$ either both belong to $\theta(S_p)$, or both belong to $\theta(S_q)$.
This completes the proof.
\end{proof}

Since Proposition~\ref{p:S2-times-S1} implies Assertion~\ref{li:2-p_2-q}, the function $\phi:S\to \Mat{2}$ is indeed multiplicative, and this completes the proof of Theorem~\ref{t:slatt-M2-amnm}.
\end{section}

\section{Concluding remarks and questions}
It would be interesting to try and find an intrinsic condition on a semilattice which is necessary and sufficient for the existence of some weight $\omega$ such that $\lom(S)$ is not AMNM. We have seen (Example~\ref{eg:cases}\ref{li:two} and Theorem~\ref{t:is-AMNM}) that $b(S)=+\infty$ is a necessary condition;
it may also be a sufficient condition, although we have not investigated further.

As remarked after the proof of Theorem~\ref{t:wt-Nmin-M2}, one can build on Theorem~\ref{t:slatt-M2-amnm} to show that $(\lom(\Nmin),\Mat{2})$ is  an AMNM pair if the weight satisfies $\sup_n \min(\om(n),\om(n+1))<\infty$. Details will be given in forthcoming work, which also plans to address  AMNM pair problems for $(\ell^1(S),B)$, with $B$ an arbitrary Banach algebra. There is some evidence to suggest that the method used for $B=\Mat{2}$ can be extended to $B=\Mat{n}$, $n\geq 3$, although a more laborious case-by-case analysis would be required. (However, see Remark~\ref{r:belated}.)

Finally, we close with the following question: {\it 
is $(\ell^1(S),\Bdd(E))$ AMNM for every Banach space $E$? If not, what if we restrict to the cases $E=\ell^p$ for $1<p<\infty$?}

\subsection*{Acknowledgements}
Some parts of this paper are based on work done in 2011, during a visit to the University of Newcastle upon Tyne. The author thanks the School of Mathematics and Statistics at the University for their hospitality.
Thanks also go to the anonymous referee for his or her diligent work, which caught several minor errors, led to the correct formulation of Theorem~\ref{t:wt-Nmin-M2},
and contained valuable suggestions that have improved the paper.

The work here was supported by NSERC Discovery Grant 402153-2011. 
Figure 1 was produced using Paul Taylor's \texttt{diagrams.sty} macros.


\begin{thebibliography}{10}

\bibitem{Birkhoff_3rdcorr}
{\sc G.~Birkhoff}, {\em Lattice theory}, vol.~25 of American Mathematical
  Society Colloquium Publications, American Mathematical Society, Providence,
  R.I., third~ed., 1979.

\bibitem{DalesLoy_diss}
{\sc H.~G. Dales \and\ R.~J. Loy}, {\em Approximate amenability of semigroup
  algebras and {S}egal algebras}, Dissertationes Math. (Rozprawy Mat.), 474
  (2010), p.~58.

\bibitem{Fein_alg}
{\sc J.~F. Feinstein}, {\em Strong {D}itkin algebras without bounded relative
  units}, Int. J. Math. Math. Sci., 22 (1999), pp.~437--443.

\bibitem{GhLZ_genam2}
{\sc F.~Ghahramani, R.~J. Loy, \and\ Y.~Zhang}, {\em Generalized notions of
  amen\-ability,~{II.}}, J. Funct. Anal., 254 (2008), pp.~1776--1810.

\bibitem{HewZuck_TAMS56}
{\sc E.~Hewitt \and\ H.~S. Zuckerman}, {\em The {$l_1$}-algebra of a commutative
  semigroup}, Trans. Amer. Math. Soc., 83 (1956), pp.~70--97.

\bibitem{Howey_PhD}
{\sc R.~A.~J. Howey}, {\em Approximately multiplicative maps between some
  {B}anach algebras}, PhD thesis, University of Newcastle upon Tyne, 2000.

\bibitem{Howey_JLMS}
\leavevmode\vrule height 2pt depth -1.6pt width 23pt, {\em Approximately
  multiplicative functionals on algebras of smooth functions}, J. London Math.
  Soc. (2), 68 (2003), pp.~739--752.

\bibitem{Jarosz_LNM}
{\sc K.~Jarosz}, {\em Perturbations of {B}anach algebras}, vol.~1120 of Lecture
  Notes in Mathematics, Springer-Verlag, Berlin, 1985.

\bibitem{Jarosz_AMNM}
\leavevmode\vrule height 2pt depth -1.6pt width 23pt, {\em Almost
  multiplicative functionals}, Studia Math., 124 (1997), pp.~37--58.

\bibitem{BEJ_AMNM1}
{\sc B.~E. Johnson}, {\em Approximately multiplicative functionals}, J. London
  Math. Soc. (2), 34 (1986), pp.~489--510.

\bibitem{BEJ_AMNM2}
\leavevmode\vrule height 2pt depth -1.6pt width 23pt, {\em Approximately
  multiplicative maps between {B}anach algebras}, J. London Math. Soc. (2), 37
  (1988), pp.~294--316.

\bibitem{Law_amnm-to-Mn}
{\sc J.~Lawrence}, {\em The stability of multiplicative semigroup homomorphisms
  to real normed algebras. {I}}, Aequationes Math., 28 (1985), pp.~94--101.

\bibitem{MO100971}
{\sc N.N.}, {\em personal communication}.
\newblock MathOverflow, 2012.\newline
\newblock \url{http://mathoverflow.net/questions/100971} (version: 2012-06-29).

\bibitem{Sidney_AMNM}
{\sc S.~J. Sidney}, {\em Are all uniform algebras {AMNM}?}, Bull. London Math.
  Soc., 29 (1997), pp.~327--330.

\end{thebibliography}
\newcommand{\and}{{\rm and}}

\vfill
\noindent
Department of Mathematics and Statistics

\noindent
University of Saskatchewan

\noindent
McLean Hall, 106 Wiggins Road

\noindent
Saskatoon, SK, Canada S7N 5E6

\noindent
Email: \texttt{choi@math.usask.ca}

\end{document}